\documentclass[11pt,reqno]{amsart}
\usepackage{bbm, amssymb,amsmath,graphicx, times,color, multicol, amsthm,url,bigints,enumitem,delarray,verbatim,soul,float,soul,longtable, comment,array,booktabs,arydshln,amsaddr}

\usepackage[authoryear, round]{natbib}
\usepackage[hidelinks]{hyperref}
\usepackage{graphicx}
\usepackage{caption}
\usepackage{subcaption}

\usepackage{pdflscape}
\usepackage{srcltx}
\usepackage{grffile}
\usepackage[hmargin=0.8in,height=8.8in]{geometry}

\newcommand{\1}{\mbox{1}\hspace{-0.25em}\mbox{l}}
\newcommand {\diff}{{\rm d}}

\newcommand {\I}{\mathcal{I}}

\newtheorem{proposition}{Proposition}
\newtheorem{remark}{Remark}[section]
\newtheorem{lemma}{Lemma}[section]

\newtheorem{assump}{Assumption}

\newtheorem{theorem}{Theorem}
\newcommand {\R}{\mathbb{R}}

\newcommand {\F}{\mathcal{F}}

\newcommand {\p}{\mathbb{P}}

\newcommand {\G}{\mathcal{G}}
\newcommand {\B}{\mathcal{B}}
\newcommand {\E}{\mathbb{E}}

\newcommand{\conn}{\quad\text{and}\quad}
\newcommand{\orr}{\;\text{or}\;}

\newcommand{\nn}{\nonumber}

\title{On  Decomposition of the Last Passage Time of Diffusions}

\author{Masahiko Egami$^1$ and Rusudan Kevkhishvili$^2$}
\address{$^{1,2}$Graduate School of Economics, Kyoto University, Sakyo-ku, Kyoto, 606-8501, Japan}
\email{egami@econ.kyoto-u.ac.jp}
\email{kevkhishvili.rusudan.2x@kyoto-u.ac.jp}

\thanks{$^1$Phone: +81-75-753-3430. $^2$Phone: +81-75-753-3429.\\
This version: \today. The first author is in part supported by Grant-in-Aid for Scientific Research (C) No.23K01467, Japan Society for the Promotion of Science. The second author was in part supported by JSPS KAKENHI Grant-in-Aid for Early-Career Scientists No.21K13324 and No.23K12501.}

\begin{document}
\maketitle

\begin{abstract}
	For a regular transient diffusion, we provide a decomposition of its last passage time to a certain state $\alpha$. This is accomplished by transforming the original diffusion into two diffusions using the occupation time of the area above and below $\alpha$. Based on these two processes, both having a reflecting boundary at $\alpha$, we derive the decomposition formula of the Laplace transform of the last passage time explicitly in a simple form in terms of Green functions. This equation also leads to the Green function's decomposition formula. We demonstrate an application of these formulas to a diffusion with two-valued parameters.
\end{abstract}



\section{Introduction}
This paper provides a decomposition of the last passage time's Laplace transform and the Green function for a general one-dimensional regular transient diffusion. Considering the last passage time to a certain state $\alpha$, the proof of the main result in Proposition \ref{prop:laplace-product} is based on the transformation of the original diffusion into two diffusions using the occupation time of the area above and below $\alpha$. To the best of our knowledge, the related Lemmas \ref{lem:exponential}-\ref{lem:lambdaB}, which are the foundations of Proposition \ref{prop:laplace-product}, are fully original.  They also provide new insights on the occupation and local times of these two diffusions since we handle two local times together in analyzing a killing time and a last passage time.  An immediate and important application of this result is Theorem \ref{theorem}, the decomposition of the Green function, the latter being one of the fundamental objects in applied mathematics (e.g. differential equations \citep{green_duffy} and potential theory including its probabilistic approach \citep{potential_doob,chung-zhao,pinsky1995}). The decomposition can be done easily as demonstrated in Section \ref{sec:OU} where we handle the Ornstein-Uhlenbeck (OU) process: its Green function involves non-elementary hard-to-treat functions. 

The decomposition formulas in Proposition \ref{prop:laplace-product} and Theorem \ref{theorem} are new results. With these formulas, the behavior of diffusions above and below a certain point $\alpha$ can be analyzed separately from the original diffusion. One example is to apply this decomposition to a diffusion whose parameters are different above and below $\alpha$. We demonstrate this point in Section \ref{sec:appl}: our results allow us to bypass the need of knowing the explicit transition density of such diffusions by reducing the original problem to the case of two non-switching diffusions. This feature is particularly important because the transition density (and its Laplace transform) in the case of switching parameters is often unavailable. Let us point out that there is no general established method in the literature for explicitly obtaining the Green function of diffusions with switching parameters. We provide this method.
In the special case of a Brownian motion with two-valued drift, \citet{benes1980} derives its Green function using the symmetry of the Brownian motion, the forward Kolmogorov equation (satisfied by the transition density), and a linear system of equations based on various conditions satisfied by the density's Laplace transform. Section \ref{sec:parameter-switch} shows that the decomposition formula saves these computations. Moreover, a diffusion with switching parameters is useful in modeling real-life problems.  For example, in Section \ref{sec:financial-appl}, we show the last passage time distribution of such a process, quantifying the leverage effect of high volatility stock.   In addition, Proposition \ref{prop:killing-rate} derives the killing rate for the diffusion above level $\alpha$ explicitly. This is also a new finding that uncovers a connection between the component diffusions in the decomposition formula.

The literature for the last passage time (or the last exit time) includes \cite{Doob1957}, \cite{nagasawa}, \cite{kunita-watanabe1966}, \cite{salminen1984},  \cite{rw-1}, \cite{chung-walsh}, \cite{revuz-yor} as well as the studies referred therein. This object is closely related to the concepts of transience/recurrence, Doob's $h$-transform, time-reversed process, and the Martin boundary theory, and has been an important subject in the probability literature. \citet{salminen1984} derives the distribution of the last passage time using the transition density of the original diffusion, which leads to its Laplace transform in terms of the original Green function \citep[Chapter II.3.20]{borodina-salminen}. See also \citet{egami-kevkhishvili-reversal}.  \emph{In contrast to the existing literature, the study of this paper is the first one to investigate the distribution of the last passage time to $\alpha$ by focusing on the regions above and below $\alpha$ separately.} Propositions \ref{prop:laplace-product}-\ref{prop:all} and Theorem \ref{theorem} characterize the behavior of the original process in these two regions, and the decomposition formulas represent a new tool for further investigation of diffusions.

A wide range of applications of last passage times  in financial modeling are discussed in \cite{nikeghbali-platen}. These applications cover the analysis of default risk, insider trading, and option valuation, which we summarize below. \cite{elliott2000} and \cite{jeanblanc_rutkowski} discuss the valuation of defaultable claims with payoff depending on the last passage time of a firm's value to a certain level. See also \cite{coculescu2012} and Chapters 4 and 5 in \cite{jeanblanc2009}.
\cite{egami-kevkhishvili-reversal} develops a new risk management framework for companies based on the last passage time of a leverage ratio to some alarming level. They derive the distribution of the time interval between the last passage time and the default time. Their analysis of company data demonstrates that the information regarding this time interval together with the distribution of the last passage time is useful for credit risk management. To distinguish the information available to a regular trader versus an insider, \cite{imkeller2002} uses the last passage time of a Brownian motion driving a stock price process. The last passage time, which is not a stopping time to a regular trader, becomes a stopping time to an insider by utilizing progressive enlargement of filtration. This study illustrates how additional information provided by the last passage time can create arbitrage opportunities. Last passage times have also been used in the  European put and call option pricing. The related studies are presented in \cite{profeta2010}. These studies show that option prices can be expressed in terms of probability distributions of last passage times. See also \cite{cheridito2012}.

The structure of the paper is the following. In the rest of this section, we summarize some mathematical facts of one-dimensional diffusion. Section \ref{sec:proof} is devoted to the proof of Proposition \ref{prop:laplace-product} and the identification of the associated killing rate. Section \ref{sec:example} is an example of last passage time decomposition. We present extensions and applications in Sections \ref{sec:Green} and \ref{sec:appl} where the decomposition for the Green function is established in a general setting (Section \ref{sec:Green}) and diffusions with switching parameters are studied in Sections \ref{sec:parameter-switch} and \ref{sec:financial-appl}, the latter being a financial application.

\subsection{Mathematical Setup}\label{sec:setup}
We refer to \citet[Chapter II]{borodina-salminen},  \citet[Chapter 15] {karlin-book}, \citet[Chapter 5]{karatzas}, \citet[Chapter 4]{IM1974}, and \citet[Chapter III]{rw-1} for diffusion processes. The main reference is the first one. Except for the proof of \eqref{lemma:psi-phi}, the facts regarding diffusions mentioned in this subsection can be found in the references above. We cite specific references for the facts that are not listed in \citet[Chapter II]{borodina-salminen}.

Let us consider a complete probability space $(\Omega,\F,\p)$
with a filtration $\mathbb{F}=(\F_t)_{t\ge 0}$ satisfying the usual conditions. Let $X$ be a regular diffusion process adapted to $\mathbb{F}$ with the state space $\mathcal{I}=(\ell,r)\subset\R$. We assume $X$ is not killed in the interior of $\mathcal{I}$, which is a standard grand assumption for a general study of regular diffusions (e.g., see \citet{salminen1984} and \citet{DK2003}). On the other hand, if $X$ hits $\ell$ or $r$, it is killed and immediately transferred to the cemetery $\Delta\notin \mathcal{I}$. The lifetime of $X$ is given by
\begin{equation*}
	\xi=\inf\{t: X(t-)=\ell \orr r\}.
\end{equation*}

Following \citet[Chapter III]{rw-1}, we write
\[X=(\Omega, \{\F_t: t\ge 0\}, \{X_t: t\ge 0\}, \{P_t: t\ge 0\}, \{\p^x: x \in \mathcal{I}\})\]
where $\p^x$ denotes the probability law of the process when it starts at $x\in \mathcal{I}$. For every $t\ge 0$, the transition function is given by $P_t: \mathcal{I}\times \B(\mathcal{I})\mapsto [0,1]$ such that for all $t,s\ge 0$ and every Borel set $A\in\B(\mathcal{I})$
\begin{equation*}
	\p^x\left(X_{t+s}\in A\mid \F_s\right)=P_t(X_s,A), \quad \p^x\textnormal{-a.s}.
\end{equation*}
\indent The dynamics of a one-dimensional diffusion are characterized by scale function, speed measure, and killing measure (see Appendix \ref{app:elements} for definitions).
The scale function and the speed measure of $X$ are given by $s(\cdot)$ and $m(\cdot)$, respectively.  The killing measure is given by $k(\cdot)$. The assumption we made above that killing does not occur in the interior of the state space is expressed by $k(\diff x)=0$ for $x\in \mathcal{I}$.

We assume that $X$ is transient. The transience is equivalent to one or both of the boundaries being attracting; that is, $s(\ell)>-\infty$ and/or $s(r)<+\infty$. See Proposition 5.22 in \citet[Chapter 5]{karatzas} and \citet{salminen1984}. Note that $s(\ell):=s(\ell+)$ and $s(r):=s(r-)$.
 Transient diffusion can also be obtained from originally recurrent diffusion (such as Brownian motion and Ornstein-Uhlenbeck process) by including a killing boundary in its state space. Such setup is often used in engineering, economics, finance, and other scientific fields when dealing with real-life problems. For example, we refer the reader to \citet{linetsky2015} for financial engineering applications such as derivative pricing. Applications in neuroscience are discussed in \citet{bibbona2013}. For optimal stopping problems, refer to \citet{alvarez-matomaki2014}. Hence transient diffusions are useful in modeling.

To obtain concrete results, we set a specific assumption:
\begin{assump}\normalfont\label{assump-s}
	\begin{equation*}\label{eq:assumption}
		s(\ell)>-\infty\conn s(r)=+\infty.
	\end{equation*}
\end{assump}
\noindent Then, it holds that
\begin{equation*}
	\p^x\left(\lim_{t\to\xi}X_t=\ell\right)=1, \quad \forall x\in \mathcal{I}
\end{equation*}
(see Proposition 5.22 in \citet[Chapter 5]{karatzas}). That is, killing occurs at $\ell$.
For the later reference, we state the definition of the killing rate of a diffusion:
the infinitesimal killing rate $\gamma(x)$ at $x\in \mathcal{I}$ is
\begin{equation}\label{eq:rate-def}
	\gamma(x):=\lim_{s\downarrow 0}\frac{1}{s}\left(1-\p^x(\xi>s)\right).
\end{equation}
Assumption \ref{assump-s} is necessary to fix a method to prove Proposition \ref{prop:laplace-product}.  But we shall remove this assumption in Proposition \ref{prop:all}.

For every $t>0$ and $x\in \mathcal{I}$, $P_t(x,\cdot):A\mapsto P_t(x,A)$ is absolutely continuous with respect to the speed measure $m$:
\begin{equation*}
	P_t(x,A)=\int_Ap(t;x,y)m(\diff y), \quad A\in\B(\mathcal{I}).
\end{equation*}
As discussed in \citet[Chapter 4.11]{IM1974}, the transition density $p$ may be constructed to be positive and jointly continuous in all variables as well as symmetric satisfying $p(t;x,y)=p(t;y,x)$.

We use superscripts $+$ and $-$ to denote the right and left derivatives of some function $f$ with respect to the scale function:
\begin{equation}\label{eq:right-left-derivative}
	f^+(x):=\lim_{h\downarrow 0}\frac{f(x+h)-f(x)}{s(x+h)-s(x)}, \quad f^-(x):=\lim_{h\downarrow 0}\frac{f(x)-f(x-h)}{s(x)-s(x-h)}.
\end{equation}
The infinitesimal generator $\G$ is defined by
\begin{equation}\label{eq:original-G}
\G f: =\lim_{t\downarrow 0}\frac{P_t f-f}{t}
\end{equation} applied to bounded continuous functions $f$ defined in $\mathcal{I}$ for which the limit exists pointwise, is a bounded continuous function in $\mathcal{I}$, and $\sup_{t>0}||\frac{P_t f-f}{t}||<\infty$ with the sup norm $||\cdot||$. We assume $s$ and $m$ are absolutely continuous with respect to the Lebesgue measure and have smooth derivatives. With this assumption together with a continuous second derivative of $s$,
the generator $\G$ coincides with the second-order differential operator given by
\begin{eqnarray}\label{eq:diff-operator}
	\G f(x)=\frac{1}{2}\sigma^2(x)f''(x)+\mu(x)f'(x), \quad x\in \mathcal{I}
\end{eqnarray}
where $\mu(\cdot)$ and $\sigma(\cdot)$ denote infinitesimal drift and diffusion parameters, respectively. We assume $\sigma^2(x)>0$ for all $x\in\mathcal{I}$. To ensure that $\diff X_t=\mu(X_t)\diff t+\sigma(X_t)\diff W_t$ (with a standard Brownian motion $W$) has a weak solution, we impose a standard condition on $\mu$ and $\sigma$:
\[\forall x\in\mathcal{I}, \; \exists\varepsilon>0 \; \textnormal{such that} \; \int_{x-\varepsilon}^{x+\varepsilon}\frac{1+|\mu(y)|}{\sigma^2(y)}\diff y<\infty.\]
See \citet[Chapter 5, Theorem 5.15]{karatzas}.
Consider the equation $\G u = q u$ for $q>0$.  Under the original definition of $\G$ in \eqref{eq:original-G}, it should read as follows: $u$ is a function which satisfies
\[
q\int_{[a, b)}u(x)m(\diff x)=u^{-}(b)-u^{-}(a)%-\int_{[a, b)}u(x)k(\diff x)
\] for all $a, b$ such that $\ell <a<b<r$. But in the absolute continuous case, $u$ is the solution to $\G u=q u$ for $\G$ in \eqref{eq:diff-operator}, so that the existence of $u$ is part of the definition of the generator.
From the generator equation we have
\begin{equation}\label{eq:scale-speed}
s(x)=\int^xe^{-\int^y\frac{2\mu(u)}{\sigma^2(u)}\diff u}\diff y, \quad m(\diff x)=\frac{2e^{\int^x\frac{2\mu(u)}{\sigma^2(u)}\diff u}}{\sigma^2(x)}\diff x.
\end{equation}
Note that such definitions and assumptions for the scale function and speed measure are used in \citet[ Chapter 5]{karatzas} and \citet[Chapter 15]{karlin-book} and that $s(x)$ satisfies $\G s=0$ on $\I$.

The Laplace transform of the hitting time $H_z:=\inf\{t\ge 0:X_t=z\}$ for $z\in \mathcal{I}$ is given by
\begin{equation}\label{eq:hitting-time-laplace}
	\E^x\left[e^{-qH_z}\right]=\begin{cases}
				\frac{\phi_q(x)}{\phi_q(z)}, \quad x\ge z,\\
        \frac{\psi_q(x)}{\psi_q(z)}, \quad x\le z,
	\end{cases}
\end{equation}
where the continuous positive functions $\psi_q$ and $\phi_q$ denote linearly independent solutions of the ODE $\G f=qf$ with $q>0$. Here $\psi_q$ is increasing while $\phi_q$ is decreasing. They are unique up to a multiplicative constant, once the boundary conditions at $\ell$ and $r$ are specified. Finally, the \emph{Green function} is defined as
\begin{equation}\label{eq:green-q}
	G_q(x,y):=\begin{cases}
		\frac{\psi_q(y)\phi_q(x)}{w_q}, \quad x\ge y, \\
\frac{\psi_q(x)\phi_q(y)}{w_q}, \quad x\le y,
	\end{cases}
\end{equation}
with the \emph{Wronskian} $w_q:=\psi_q^+(x)\phi_q(x)-\psi_q(x)\phi_q^+(x)=\psi_q^-(x)\phi_q(x)-\psi_q(x)\phi_q^-(x)$. It holds that $G_q(x,y)=\int_{0}^{\infty}e^{-qt}p(t;x,y)\diff t$ for $x,y\in \mathcal{I}$.

Under Assumption \ref{assump-s}, the killing boundary $\ell$ is attracting and $\lim_{x\downarrow\ell}\E^x\left[e^{-qH_z}\right]=\frac{\psi_q(\ell+)}{\psi_q(z)}=0$ for $z\in \mathcal{I}$. Hence $\psi_q(\ell+)=0$. As the right boundary $r$ is not attracting, $\lim_{z\uparrow r}\E^x\left[e^{-qH_z}\right]=\frac{\psi_q(x)}{\psi_q(r-)}=0$ for $x\in \mathcal{I}$ and we obtain $\psi_q(r-)=+\infty$.

Next, due to the transience of $X$, we define
\begin{equation}\label{eq:G0}
	G_0(x,y):=\lim_{q\downarrow 0}G_q(x,y)=\int_{0}^{\infty}p(t;x,y)\diff t<+\infty.
\end{equation}
Following \citet[Section 4.11]{IM1974}, this quantity is represented by
\begin{equation}\label{eq:green-0}
	G_0(x,y)=\begin{cases}
\frac{\psi_0(y)\phi_0(x)}{w_0}, \quad x\ge y, \\
		\frac{\psi_0(x)\phi_0(y)}{w_0}, \quad x\le y,
			\end{cases}
\end{equation}
where the continuous positive functions $\psi_0$ and $\phi_0$ denote (linearly independent) solutions of the ODE $\G f=0$ and
\begin{equation*}\label{eq:w0}
w_0:=\psi_0^+(x)\phi_0(x)-\psi_0(x)\phi_0^+(x)=\psi_0^-(x)\phi_0(x)-\psi_0(x)\phi_0^-(x).
\end{equation*}
Here $\psi_0$ is increasing while $\phi_0$ is decreasing. These functions are uniquely determined based on the boundary conditions and satisfy
\begin{equation}\label{lemma:psi-phi}
\phi_0\equiv1, \qquad
\psi_0(\ell+)=0, \qquad
\psi_0(r-)=+\infty
\end{equation}
under Assumption \ref{assump-s}. We show a proof of \eqref{lemma:psi-phi} in Appendix \ref{app:proof-phi-psi} since to our knowledge it is not shown in the existing literature.   The first equation $\phi_0\equiv 1$ should be understood such that the solution $\phi_0$ can be taken as unity.

Since $\psi_0$ solves $\G f=0$ and is increasing, we can set $\psi_0(x)=w_0 (s(x)+\text{constant})$.  Then, the boundary condition at $\ell$ determines the constant, i.e.,
\[\psi_0(x)=w_0(s(x)-s(\ell)), \quad x\in \mathcal{I},\]
which in turn leads to
\begin{equation}\label{eq:G0-explicit}
	G_0(x,y)=(s(x)-s(\ell))\wedge (s(y)-s(\ell)),
\end{equation} since $\phi_0\equiv 1$ by \eqref{lemma:psi-phi}. Note that $w_0$ is forced to be $\psi_0^-$.

Our analysis focuses on a decomposition of the last passage time of some fixed level $\alpha\in \mathcal{I}$ which is denoted by
\begin{equation}\label{eq:lambda}
	\lambda_\alpha:=\sup\{t:X_t=\alpha\}
\end{equation}
with $\sup\emptyset=0$. Our objective is to decompose the Laplace transform of $\lambda_\alpha$ in a simple formula convenient for use (see Proposition \ref{prop:laplace-product}).
As $X$ is a transient diffusion, $\lambda_\alpha<+\infty$ a.s. The distribution of the last passage time is given by
\begin{equation*}\label{eq:lambda-dist}
	\p^x(0<\lambda_\alpha\le t)=\int_0^t\frac{p(u;x,\alpha)}{G_0(\alpha,\alpha)}\diff u.
\end{equation*}
See \citet[Chapter II.3.20]{borodina-salminen}, \citet[Proposition 4]{salminen1984}, \citet{egami-kevkhishvili-reversal}. Then, the Laplace transform is
\begin{equation}\label{eq:lambda-laplace}
	\E^x\left[e^{-q\lambda_\alpha}\right]=\int_{0}^{\infty}e^{-qt}\frac{p(t;x,\alpha)}{G_0(\alpha,\alpha)}\diff t=\frac{G_q(x,\alpha)}{G_0(\alpha,\alpha)}, \quad x\ge\alpha.
\end{equation}
To the best of our knowledge, these two formulas are essentially the known results about last passage time distribution.   Our objective is to study its decomposition by focusing on the regions above and below $\alpha$ separately. This results in a simple decomposition formula connecting two diffusions with reflecting boundaries at $\alpha$. We derive this new result in Section \ref{sec:proof}.

\section{Decomposition of $X$ and its last passage time}\label{sec:proof}
Consider the transient diffusion $X$ on $\mathcal{I}=(\ell, r)$ with Assumption \ref{assump-s} and its last passage time $\lambda_\alpha$ of some fixed level $\alpha\in \mathcal{I}$ as defined in \eqref{eq:lambda}. We draw a schematic diagram of the dynamics of $X$ along with $\lambda_\alpha$ in Figure \ref{fig:whole_path}.

\subsection{Time-changed processes}\label{sec:time-changed process}
Let us fix some $\alpha\in \mathcal{I}$ and consider an occupation time of the region above and below $\alpha$
\begin{equation*}\label{eq:occup_time}
	\Gamma_+(t):=\int_{0}^{t}\mathbf{1}_{\{X_s\ge\alpha\}}\diff s \conn
	\Gamma_-(t):=\int_{0}^{t}\mathbf{1}_{\{X_s<\alpha\}}\diff s
\end{equation*}
together with its right inverse:
\begin{equation*}\label{eq:inverse_occup}
	\Gamma_+^{-1}(t):=\inf\{s:\Gamma_+(s)>t\} \conn
	\Gamma^{-1}_-(t):=\inf\{s: \Gamma_-(s)>t\}.
\end{equation*}
Define $\lambda^A_\alpha:=\Gamma_+(\lambda_\alpha)$ and $\lambda^B_\alpha:=\Gamma_-(\lambda_\alpha)$.
Then, it holds that
\begin{equation*}\label{eq:add-lambda}
	\lambda_\alpha=\lambda^A_\alpha + \lambda^B_\alpha.
\end{equation*}
Throughout the paper,
\begin{center}
superscript $A$ stands for ``\emph{above} the level $\alpha$"  and 
superscript $B$ stands for ``\emph{below} the level $\alpha$".
\end{center}
We will use the following time-changed processes:
\begin{equation}\label{eq:X-time-change}
	\hat{X}^A(t):=X(\Gamma_+^{-1}(t))\conn X^B(t):=X(\Gamma_-^{-1}(t)).
\end{equation}
Time-change by additive functionals (such as occupation time) are discussed in \citet[Chapter 15.8]{karlin-book}. See also \citet{rogers1991} and \citet{walsh1978} where such techniques are employed for the path decomposition. Our interest lies in how we can decompose the last passage time $\lambda_\alpha$ of $X$ to the point $\alpha$ in terms of $\hat{X}^A$ and $X^B$.
	
Note that $\hat{X}^A$ and $X^B$ have the same speed measure and scale function as $X$ \citep[Theorem 10.12]{dynkin1}. Then, $X^B$ can be seen as the process for which $\alpha$ is a reflecting boundary. Similarly, $\hat{X}^A$ can be considered as the process for which $\alpha$ is an elastic boundary. For the definition of an elastic boundary, refer to \citet[Chapter II.1.7]{borodina-salminen}. The hat is to stress that $\hat{X}^A$ is a killed process with the non-zero killing measure $\hat{k}^A$.
See Figure \ref{fig:picture} for this probabilistic feature.

\begin{figure}[H]	
	\begin{subfigure}{0.75\textwidth}
		\begin{center}
\hspace*{-7cm}			
  \includegraphics[scale=0.7]{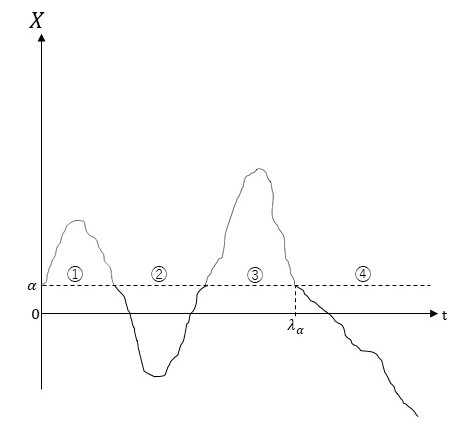}
			\subcaption{The original process $X$ and its last passage time $\lambda_\alpha$.}
			\label{fig:whole_path}
		\end{center}
	\end{subfigure}
	\begin{subfigure}{0.75\textwidth}
		\hspace*{-5cm}	\includegraphics[scale=0.7]{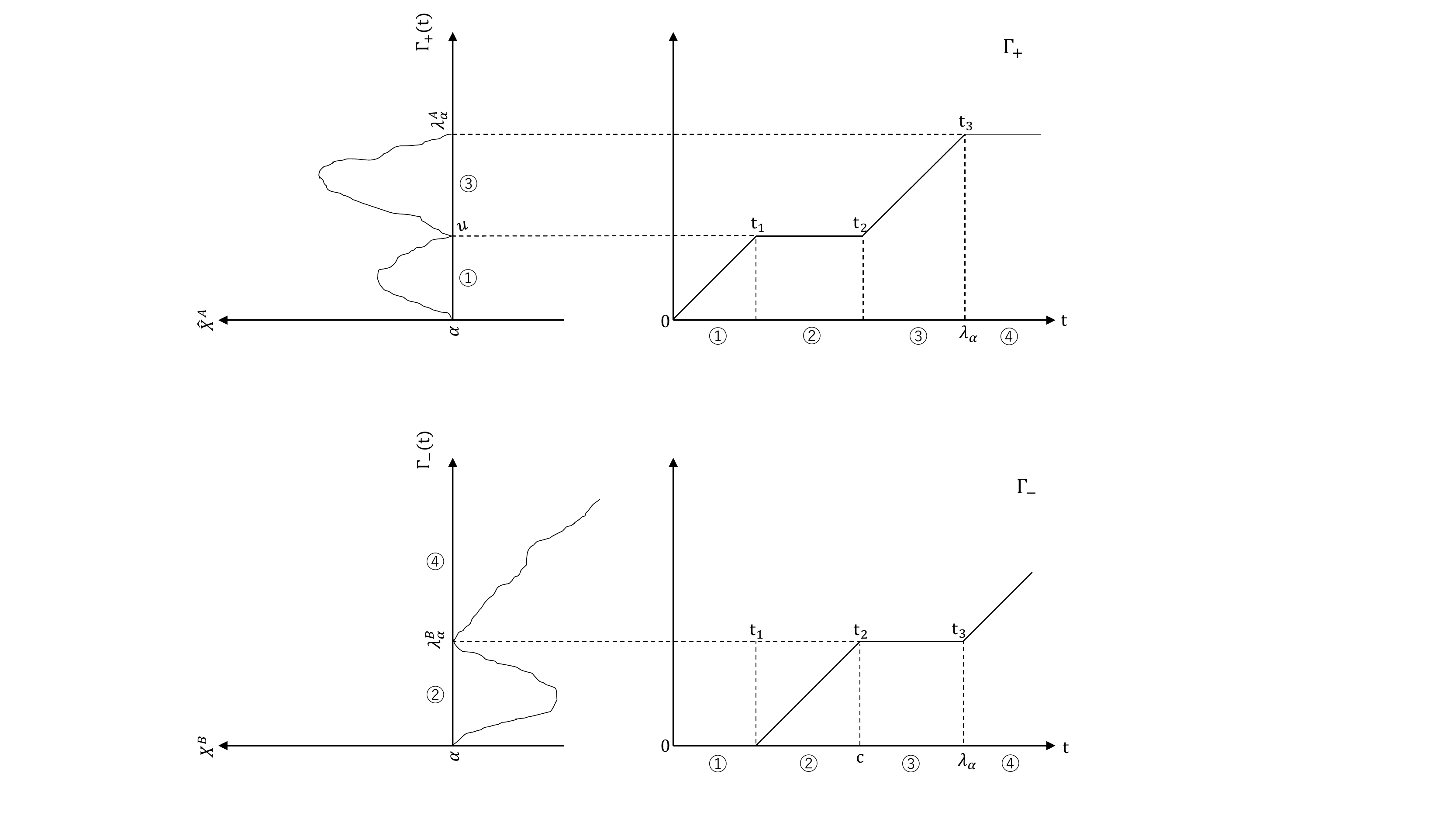}
		\subcaption{A schematic expression of $\hat{X}^A$ and $X^B$. Note that the graphs of $\hat{X}^A$ and $X^B$ are presented by rotating $90^\circ$ degrees counter-clockwise. }
		\label{fig:picture}
	\end{subfigure}
	\caption{}
\end{figure}
Further note that
\[\hspace{-1cm}\text{$\lambda_\alpha^A$ is considered as the killing time of $\hat{X}^A$ and $\lambda_\alpha^B$ is the last passage time of $X^B$ to level $\alpha$}.
\]

Let us introduce another diffusion $X^A$ on $[\alpha, r)$ for which $\alpha$ is a reflecting boundary. The process $X^A$ has the same speed measure and scale function as $\hat{X}^A$ (hence the same  as $X$) but its killing measure is \emph{zero}. To be precise, we denote
\begin{equation}\label{eq:Xhat}
	\hat{X}^A(t)=\begin{cases}
		X^A(t), \quad 0\le t <\lambda_\alpha^A\\
		\Delta, \quad t\ge \lambda_\alpha^A
	\end{cases}
\end{equation} to distinguish the killed process $\hat{X}^A$ from $X^A$. Even after time $\lambda_\alpha^A$, the scale function, speed measure and the zero killing measure of $X^A$ are unchanged and $X^A$ continues to be reflected at $\alpha$. The killing measure of $\hat{X}^A$ satisfies $\hat{k}^A(\{\alpha\})>0$ and $\hat{k}^A(\diff x)=0$ for $x\neq \alpha$. Note that the speed measure $\hat{m}^A$ of $\hat{X}^A$ satisfies $\hat{m}^A(\{\alpha\})=0$.

Let us summarize the above construction for later use:
\begin{remark}\normalfont\label{rem:three-processes}
  Under Assumption \ref{assump-s}, $X^A$ is a diffusion on $\mathcal{I}^A:=[\alpha, r)$  with zero killing measure in $\mathcal{I}^A$. $\hat{X}^A$ is defined by \eqref{eq:X-time-change} on $\mathcal{I}^A$ with a non-zero killing measure in $\mathcal{I}^A$ and $X^B$ is defined by \eqref{eq:X-time-change} on $\mathcal{I}^B:=(\ell, \alpha]$ with zero killing measure in $\mathcal{I}^B$. All the three processes have the reflecting point at $\alpha$ and their scale functions and speed measures are the same as those of the original process $X$. The decomposition of the original process $X$ into $\hat{X}^A$ and $X^B$ is illustrated in Figure \ref{fig:picture}.
\end{remark}

We begin with the case where $X$ starts at $\alpha$, so that $\lambda_\alpha>0$. This assumption is relaxed in Theorem \ref{theorem} in Section \ref{sec:Green}, where we treat explicitly the possibility of $\lambda_\alpha=0$. The main result shows that the Laplace transform of the last passage time $\lambda_\alpha$ can be decomposed into two parts:
\begin{proposition}\label{prop:laplace-product}
	Under Assumption \ref{assump-s}, the Laplace transform of $\lambda_\alpha$ in \eqref{eq:lambda} is represented as
	\begin{equation}\label{eq:prop-laplace}
		\E^\alpha[e^{-q\lambda_\alpha}]=\frac{G_q^A(\alpha, \alpha)}{G^A_q(\alpha, \alpha)+G^B_q(\alpha, \alpha)}\cdot \frac{G^B_q(\alpha, \alpha)}{G^B_0(\alpha, \alpha)}
	\end{equation}
	where $G^A_\cdot(\cdot, \cdot)$ and $G^B_\cdot(\cdot, \cdot)$ are the Green functions of $X^A$ (not $\hat{X}^A$) and $X^B$, respectively.
\end{proposition}
\begin{proof}
	The series of Lemmas \ref{lem:exponential}-\ref{lem:lambdaB} in the next subsection lead to this result.
\end{proof}

Before we start with the lemmas, we need to introduce the local time at $\alpha$ for $\hat{X}^A$ and $X^B$ by denoting
\begin{equation}\label{eq:local-time-definition}
\hat{L}^A(t):=\hat{L}^A(t, \alpha)=L(\Gamma^{-1}_+(t), \alpha)\conn L^B(t):=L^B(t, \alpha)=L(\Gamma^{-1}_-(t), \alpha),
\end{equation} where $L(\cdot, \alpha)$ is the local time of $X$ at $\alpha$.  Let us also define the inverse local time processes
\begin{equation}\label{eq:ro-def}
\hspace{-1cm}
\hat{\rho}^A(s):=\hat{\rho}^A(s, \alpha)=\inf\{t: \hat{L}^A(t, \alpha)>s\} \conn \rho^B(s):=\rho^B(s, \alpha)=\inf\{t: L^B(t, \alpha)>s\}.
\end{equation}
For a standard Brownian motion, local times such as $\hat{L}^A$ and $L^B$ are referred to as intrinsic local time (see \citet{rogers1991}).
\begin{remark}\label{rem:indep}\normalfont
	Due to the Markov property of $X$, the excursions of $X^A$ from $\alpha$ are independent of the excursions of $X^B$ from $\alpha$: in particular during $[0, \lambda_\alpha^A)$. Recall that $X^A(t)=\hat{X}^A(t)$ and $L^A(t)=\hat{L}^A(t)$ on $[0, \lambda_\alpha^A)$ where $L^A(t)$ is the local time at $\alpha$ for $X^A$. Refer to Figure \ref{fig:picture}.  Note that the excursion of $X^B$ commencing at time $0$ by the clock $\Gamma_-(\cdot)$ occurs when $X^A$ returns to $\alpha$ at time $u$ by the clock $\Gamma_+(\cdot)$. This excursion of $X^A$ corresponds to the time interval $[0,t_1)$ in the real clock $(t)$ and is independent of $X^B$ and hence of $L^B$. Then, the excursion of $X^A$ commencing at time $u$ by the clock $\Gamma_+(\cdot)$ occurs when $X^B$ returns to $\alpha$ at time $\lambda_\alpha^B$ by the clock $\Gamma_-(\cdot)$. This excursion of $X^B$ corresponds to the time interval $[t_1,t_2)$ in the real clock $(t)$ and is independent of $X^A$ and hence of $L^A$.
	In this way, the construction of $X^A$ and $X^B$ in \eqref{eq:X-time-change} and \eqref{eq:Xhat} implies that $L^A(\cdot)$ and $L^B(\cdot)$ are independent.
\end{remark}

We use superscript $B$ to denote quantities associated with $X^B$. We denote Green functions of $X^A$ and $X^B$ by $G^A_\cdot(\cdot, \cdot)$ and $G^B_\cdot(\cdot, \cdot)$, respectively. Let us stress that $G^A_\cdot$ is \emph{not} the Green function of $\hat{X}^A$.

\subsection{The Laplace transform of $\lambda_\alpha$}
Let us introduce a (generic) exponential random variable $\mathbf{e_q}$ with rate $q>0$ which is independent of $X$.  Hence it is independent of both $\hat{X}^A$ and $X^B$. Recall \eqref{eq:Xhat} which states that $\hat{X}^A(t)=X^A(t)$ for $t\in [0, \lambda_\alpha^A)$. In this subsection, the argument is concerned with the time interval $[0, \lambda_\alpha^A)$, so that we deal with $X^A$, not $\hat{X}^A$. For simplicity, in the sequel we omit the subscript $\alpha$ to denote $\lambda:=\lambda_\alpha$,  $\lambda^A:=\lambda^A_\alpha$, and $\lambda^B:=\lambda^B_\alpha$.

Let us start with
\begin{align}\label{eq:main}
	\E^\alpha[e^{-q\lambda}]&=\p^\alpha(\lambda \le \mathbf{e_q})=\p^\alpha(\Gamma_+(\lambda)\le \Gamma_+(\mathbf{e_q}), \Gamma_-(\lambda)\le \Gamma_-(\mathbf{e_q}))\nn \\
	&=\p^\alpha(\lambda^A\le \Gamma_+(\mathbf{e_q}), \lambda^B\le \Gamma_- (\mathbf{e_q}))\nn \\
	&=\p^\alpha\left[\lambda^A\le \Gamma_+(\mathbf{e_q})\mid \lambda^B\le \Gamma_- (\mathbf{e_q})\right]\p^\alpha(\lambda^B\le \Gamma_- (\mathbf{e_q})).
\end{align}
We shall compute explicitly the right-hand side of \eqref{eq:main}.  Let us first consider the set $\{\omega: \lambda^B(\omega)\le \Gamma_-(\mathbf{e_q})(\omega)\}$.
\begin{lemma}\label{lem:exponential}
	Let $\mathbf{e_q}$ be a (generic) exponential random variable with rate $q>0$.
	Define $P:=\{\omega: \lambda^B(\omega)\le \Gamma_-(\mathbf{e_q})(\omega)\}$ and $Q:=\{\omega:\lambda^B(\omega)\le \mathbf{e_q}(\omega)\}$. Then the sets $P$ and $Q$ are equivalent.  Similarly, the sets $\{\omega: \lambda^A(\omega)\le \Gamma_+(\mathbf{e_q})(\omega)\}$ and $\{\omega:\lambda^A(\omega)\le \mathbf{e_q}(\omega)\}$ are equivalent.
\end{lemma}
\begin{proof}
	Suppose that $\omega\in P$. Then $\lambda^B\le \Gamma_-(\mathbf{e_q})\le \mathbf{e_q}$ by the definition of $\Gamma_-(\cdot)$, so that $\omega\in Q$.  On the other hand, suppose that $\omega\in Q$.  This implies that by the memoryless property
	\begin{equation}\label{eq:(1)}
		\mathbf{e_q}-\lambda^B=e'_q \circ \theta(\lambda^B)
	\end{equation} where $e'_q$ is another exponential random variable with rate $q$ and $\theta(\cdot)$ is the shift operator.  Define $J:=\Gamma^{-1}_-(\mathbf{e_q})$.  Since $\Gamma^{-1}_-(\lambda^B)=\lambda$ and $\Gamma_{-}(t)$  is strictly increasing on $t\ge \lambda$, we have
	\begin{equation}\label{eq:(2)}
		\mathbf{e_q}-\lambda^B=J-\lambda=J'\circ \theta(\lambda)
	\end{equation} for some nonnegative random variable $J'$.  From \eqref{eq:(1)} and \eqref{eq:(2)}, $J'$ is an exponential random variable with rate $q$. Then, the representation
	\begin{equation}\label{eq:J}
		J=\lambda+J'\circ \theta(\lambda)
	\end{equation} implies that it is also an exponential random variable with rate $q$. Indeed, $J$ is a continuous random variable and \eqref{eq:J} shows that $J$ has the memoryless property as $\mathbf{e_q}$ in \eqref{eq:(1)} does, so that $J$ must be an exponential random variable with rate $q$. Now, $\lambda\le J$ implies that $\lambda^B=\Gamma_-(\lambda)\le \Gamma_-(J)$.  By rewriting $J$ as a generic exponential random variable $\mathbf{e_q}$, we conclude that $\omega\in P$.
\end{proof}
The next two lemmas are concerned with the first term of \eqref{eq:main}, the conditional probability.
\begin{lemma}\label{lem:first-term}
	For the exponential random variable $\mathbf{e_q}$ with rate $q>0$, we have
	\begin{equation*}
		\p^\alpha[\lambda^A\le \Gamma_+(\mathbf{e_q})\mid \lambda^B\le \Gamma_- (\mathbf{e_q})]=\p^\alpha[L^B(\Gamma_-(\mathbf{e_q}))<L^A(\Gamma_+(\mathbf{e_q}))\mid \lambda^B\le \Gamma_- (\mathbf{e_q})].
	\end{equation*}
\end{lemma}
\begin{proof}
	Define
	\begin{equation}\label{eq:u}
		u:=\sup\{t<\lambda^A: X^A_t=\alpha\};
	\end{equation}
	that is, the last time of visit to $\alpha$ before $\hat{X}^A$ is elastically killed at time $\lambda^A=\Gamma_+(\lambda)$.  This time point $u$ is characterized as
	$\Gamma_-(\Gamma_+^{-1}(u))=\lambda^B$, and hence we have
	\begin{equation}\label{eq:LA=LB}
		L^A(u)=L^B(\lambda^B).
	\end{equation}  It may be useful to see the schematic diagram in Figure \ref{fig:picture} where the time $u$ is marked in the upper left panel.
	
	Since the local time $L^A(\cdot)$ of $X^A$ (see \eqref{eq:local-time-definition} and Remark \ref{rem:indep}) shall not increase until the next visit to $\alpha$ by $X^A$ (at time $\lambda^A$), we have
	\begin{equation}\label{eq:L^A-new-characterization}
		\lambda^A=\inf\{t: L^A(t)>L^A(u)\},
	\end{equation} which implies $L^A(\lambda^A)>L^B(\lambda^B)$ from \eqref{eq:LA=LB}.
	
	Now we condition on $\lambda^B\le \Gamma_-(\mathbf{e_q})$.  Under this condition, it is easy to see that the event  $\{\lambda^A\le \Gamma_+(\mathbf{e_q})\}$ occurs when and only when the equal sign holds.
Due to the argument in the preceding paragraph, under the condition
	$\lambda^B\le \Gamma_-(\mathbf{e_q})$, $\{\lambda^A\le \Gamma_+(\mathbf{e_q})\}=\{L^A(\Gamma_+(\mathbf{e_q}))>L^B(\lambda^B)\}$ but we also have  \[L^A(\Gamma_+(\mathbf{e_q}))=L^A(\lambda^A)>L^B(\lambda^B)=L^B(\Gamma_-(\mathbf{e_q})),\]
	which proves the lemma.
\end{proof}

\begin{lemma}\label{lem:lemma2}
	For the exponential random variable $\mathbf{e_q}$ with rate $q>0$, we have
	\begin{equation}\label{eq:lemma2}
		\p^\alpha[L^B(\Gamma_-(\mathbf{e_q}))<L^A(\Gamma_+(\mathbf{e_q}))\mid \lambda^B\le \Gamma_- (\mathbf{e_q})]=\frac{G_q^A(\alpha, \alpha)}{G^A_q(\alpha, \alpha)+G^B_q(\alpha, \alpha)}.
	\end{equation}
\end{lemma}
\begin{proof}
	First, we shall prove that the left-hand side of \eqref{eq:lemma2} simplifies to
	\begin{equation}\label{eq:interim}
		\p^\alpha[L^B(\Gamma_-(\mathbf{e_q}))<L^A(\Gamma_+(\mathbf{e_q}))\mid \lambda^B\le \Gamma_- (\mathbf{e_q})]=\p^\alpha(L^B(\mathbf{e_q})<L^A(\mathbf{e_q})).
	\end{equation}
	Indeed, given the fact  $\lambda^B\le \Gamma_-(\mathbf{e_q})$, due to Lemma \ref{lem:exponential},
	\begin{equation}\label{eq:reduceB-eq}
		\Gamma_-(\mathbf{e_q})-\lambda^B=\mathbf{e_q}\circ \theta(\lambda^B)=\mathbf{e_q}-\lambda^B.
	\end{equation}
Let us denote (see the lower right panel in Figure \ref{fig:picture})
	\[
	c:=\inf\{t: \Gamma_-(t)\ge \lambda^B\},
	\] for which the condition $\lambda^B\le \Gamma_-(\mathbf{e_q})$ implies that $c\le \mathbf{e_q}$.  Since the time point $c$ is the left-end point of a region where $\Gamma_-(\cdot)$ becomes constant, it corresponds to the left-end point of an excursion of $X^A$ from level $\alpha$. Hence $\Gamma_+(c)=u$, the right-hand side being defined in \eqref{eq:u}. Using the same argument as in Lemma \ref{lem:exponential}, we have
	\begin{equation}\label{eq:reduceA-eq}
		\Gamma_+(\mathbf{e_q})-u=\mathbf{e_q} \circ \theta(u)=\mathbf{e_q}-u.
	\end{equation}
	 Recall also \eqref{eq:LA=LB}. In other words, equations \eqref{eq:reduceB-eq} and \eqref{eq:reduceA-eq} imply that, instead of evaluating $L^B$ and $L^A$ at $\Gamma_-(\mathbf{e_q})$  and $\Gamma_+(\mathbf{e_q})$, we can evaluate $L^B$ and $L^A$ both at time $\mathbf{e_q}$. We have proved \eqref{eq:interim} using the memoryless property of $\mathbf{e_q}$.
	
	Let us now evaluate $\p^\alpha(L^B(\mathbf{e_q})<L^A(\mathbf{e_q}))$.  It is known that the random variables $L^B(\mathbf{e_q})$ and $L^A(\mathbf{e_q})$ are exponentially distributed and
	\[
	\p^\alpha(L^A(\mathbf{e_q})>s)=\p^\alpha(\rho^A(s)<\mathbf{e_q})=\E^\alpha[e^{-q\rho^A(s)}]=\exp\left(-\frac{s}{G^A_q(\alpha, \alpha)}\right)
	\] where $\rho^A(s):=\inf\{t: L^A(t, \alpha)>s\}$ is the inverse local time process for $X^A$ (c.f. \eqref{eq:ro-def}). See \citet[Section 7]{getoor1979}. Similarly, we have
	\begin{equation}\label{eq:rate-LB}
		\p^\alpha(L^B(\mathbf{e_q})>s)=\exp\left(-\frac{s}{G^B_q(\alpha, \alpha)}\right).
	\end{equation}
	Since $\mathbf{e_q}$ is independent of $X$, $L^A(\mathbf{e_q})$ and $L^B(\mathbf{e_q})$ are independent. See Remark \ref{rem:indep}. We have
	\[
	\p^\alpha(L^B(\mathbf{e_q})<L^A(\mathbf{e_q}))=\frac{\frac{1}{G^B_q(\alpha, \alpha)}}{\frac{1}{G^A_q(\alpha, \alpha)}+\frac{1}{G^B_q(\alpha, \alpha)}},
	\] which yields \eqref{eq:lemma2}.
\end{proof}
\noindent By the two lemmas, we have computed the first term of \eqref{eq:main} on its right-hand side:
\begin{equation}\label{eq:the-first-term}
	\p^\alpha[\lambda^A\le \Gamma_+(\mathbf{e_q})\mid \lambda^B\le \Gamma_- (\mathbf{e_q})]=\frac{G_q^A(\alpha, \alpha)}{G^A_q(\alpha, \alpha)+G^B_q(\alpha, \alpha)}.
\end{equation}  Let us proceed to the second term of \eqref{eq:main}.

\begin{lemma}\label{lem:lambdaB}
	It holds that for the exponential random variable $\mathbf{e_q}$ with rate $q>0$,
	\begin{equation}\label{eq:second-term}
		\p^\alpha(\lambda^B\le \Gamma_- (\mathbf{e_q}))=\frac{G^B_q(\alpha, \alpha)}{G^B_0(\alpha, \alpha)}.
	\end{equation}
\end{lemma}
\begin{proof}
	Due to Lemma \ref{lem:exponential}, $\p^\alpha(\lambda^B\le\Gamma_-(\mathbf{e_q}))=\p^\alpha(\lambda^B\le \mathbf{e_q})=\E^\alpha[e^{-q\lambda^B}]$. Using the expression of the Laplace transform of the last passage time in \eqref{eq:lambda-laplace}, we obtain \eqref{eq:second-term}.
\end{proof}

By combining Lemmas \ref{lem:first-term}-\ref{lem:lambdaB}, i.e., plugging \eqref{eq:the-first-term} and \eqref{eq:second-term} into \eqref{eq:main}, we obtain the result of Proposition \ref{prop:laplace-product}.

\subsection{The killing rate of $\hat{X}^A$}
In this subsection, we shall find the infinitesimal killing rate of $\hat{X}^A$ at $\alpha$ under the condition $\lambda_B\le \Gamma_-(\mathbf{e_q})$. We denote this rate by $\gamma_q$. Recall that the killing time for the process $\hat{X}^A$ has been denoted by $\lambda^A=\Gamma_+(\lambda)$. By \eqref{eq:rate-def}, $\gamma_q$ is given by
\begin{equation}\label{eq:gamma-q}
	\gamma_q:=\lim_{s\downarrow 0}\frac{1}{s}\left(1-\p^\alpha[\lambda^A>s\mid \lambda^B\le \Gamma_- (\mathbf{\mathbf{e_q}})]\right).
\end{equation}
For the purpose of finding $\gamma_q$, we represent the first term on the right-hand side of \eqref{eq:main} in an alternative way. More specifically, we shall prove the following:
\begin{proposition}\label{prop:killing-rate}
	The first term on the right-hand side of \eqref{eq:main} has the representation in terms of the infinitesimal killing rate $\gamma_q$
	\begin{equation}\label{eq:alternative}
		\p^\alpha[\lambda^A\le \Gamma_+(\mathbf{e_q})\mid \lambda^B\le \Gamma_- (\mathbf{e_q})]=\frac{\gamma_q\cdot G^A_q(\alpha, \alpha)}{1+\gamma_q\cdot G^A_q(\alpha, \alpha)},
	\end{equation}
	where $\gamma_q=\frac{1}{G^B_q(\alpha, \alpha)}$.
\end{proposition}
\begin{remark}\normalfont
	When we plug the value of $\gamma_q=\frac{1}{G^B_q(\alpha, \alpha)}$ into \eqref{eq:alternative}, we retrieve  \eqref{eq:the-first-term}:
	\[
	\p^\alpha[\lambda^A\le \Gamma_+(\mathbf{e_q})\mid \lambda^B\le \Gamma_- (\mathbf{e_q})]=\frac{\gamma_q\cdot G^A_q(\alpha, \alpha)}{1+\gamma_q \cdot G^A_q(\alpha, \alpha)}=\frac{G_q^A(\alpha, \alpha)}{G^A_q(\alpha, \alpha)+G^B_q(\alpha, \alpha)}.
	\]
\end{remark}
\begin{proof}
Recall that $X^A$ is a diffusion on $[\alpha, r)$ and is reflecting at $\alpha$ with its killing measure being \emph{zero}. Recall also that the local time at $\alpha$ of $X^A$ is $L^A(t)$ and $\rho^A(s)=\inf\{t: L^A(t)>s\}$ is the inverse local time process. Referring to \citet[Section 5.6]{IM1974}, one could obtain a process identical in law to $\hat{X}^A$, conditioned on $\lambda^B\le \Gamma_- (\mathbf{\mathbf{e_q}})$, by killing the process $X^A$ in the following way: let $\tau$ be an independent exponential random variable with rate $\gamma_q$ and kill $X^A$ at the time $\inf\{t:L^A(t)\ge \tau\}$.
	That is, with the definition of $\gamma_q$ in \eqref{eq:gamma-q},
	\begin{equation*}\label{eq:interim3}
		\p^\alpha[\lambda^A>s\mid \lambda^B\le\Gamma_-(\mathbf{e_q})]=\p^\alpha(\tau>L^A(s)\mid \lambda^B\le\Gamma_-(\mathbf{e_q}))=\E^\alpha[e^{-\gamma_qL^A(s)}],
	\end{equation*}
	which is equivalent to saying that
	\begin{equation}\label{eq:A-killing1}
		\p^\alpha[L^A(\lambda^A)>s\mid \lambda^B\le\Gamma_-(\mathbf{e_q})]=\p^\alpha[\lambda^A>\rho^A(s)\mid \lambda^B\le\Gamma_-(\mathbf{e_q})]=\p^\alpha(\tau> s\mid \lambda^B\le\Gamma_-(\mathbf{e_q}))=e^{-\gamma_q s},
	\end{equation}
	where the first equality is due to $\rho^A(L^A(\lambda^A))=\sup\{t: L^A(t)=L^A(\lambda^A)\}=\lambda^A$.
	
	On the other hand, recall that we have $\lambda^A=\inf\{t: L^A(t)>L^A(u)\}=\inf\{t: L^A(t)>L^B(\lambda^B)\}$ in \eqref{eq:LA=LB} and \eqref{eq:L^A-new-characterization}, with or without the condition $\lambda^B\le \Gamma_- (\mathbf{\mathbf{e_q}})$. Thus, $X^A$ is killed at the time $\inf\{t: L^A(t)>L^B(\lambda^B)\}$. It follows that the role played by $\tau$ is identical to the role played by $L^B(\lambda^B)$ under the condition $\lambda^B\le \Gamma_- (\mathbf{\mathbf{e_q}})$.
	
	Take $L^B(\lambda^B)$. Under $\lambda^B\le \Gamma_- (\mathbf{\mathbf{e_q}})$, the local time $L^B$ of $X^B$ shall not increase after time $\lambda^B$ due to the occurrence of $\lambda^B$, so that $L^B(\lambda^B)=L^B(\Gamma_-(\mathbf{e_q}))$ and therefore, we have
	\begin{align}\label{eq:A-killing2}
		\p^\alpha[L^B(\lambda^B)>s\mid \lambda^B\le\Gamma_-(\mathbf{e_q})]
		&=\p^\alpha(L^B(\Gamma_-(\mathbf{e_q}))>s\mid \lambda^B\le\Gamma_-(\mathbf{e_q})) \nn \\
		&=\p^\alpha(L^B(\mathbf{e_q})>s)=\exp\left(-\frac{s}{G^B_q(\alpha, \alpha)}\right)
	\end{align}
	where the last two equalities are due to Lemma \ref{lem:exponential} with \eqref{eq:reduceB-eq} and \eqref{eq:rate-LB}, respectively.
	It follows from the comparison of \eqref{eq:A-killing1} and \eqref{eq:A-killing2}  that $\gamma_q=\frac{1}{G^B_q(\alpha, \alpha)}$.
	
	Finally, we derive \eqref{eq:alternative}.  We have shown above that, given $\lambda^B\le \Gamma_- (\mathbf{\mathbf{e_q}})$,  the way of killing $X^A$ using the exponential random variable $\tau$ with rate $\gamma_q=\frac{1}{G^B_q(\alpha, \alpha)}$ is identical to the way in which one kills $X^A$ as in \eqref{eq:Xhat}.  This fact is used in the second equality below:
	\begin{align*}\label{eq:Gamma+}
		&\p^\alpha[\lambda^A\le \Gamma_+(\mathbf{\mathbf{e_q}})\mid \lambda^B\le \Gamma_- (\mathbf{\mathbf{e_q}})]=\p^\alpha[\lambda^A\le \mathbf{e_q} \mid \lambda^B\le \mathbf{e_q}]=\E^\alpha\left[e^{-q\inf\{t:L^A(t)\ge \tau\}}\right] \nn\\
		&=\int_0^\infty\E^\alpha\left[e^{-q\inf\{t:L^A(t)\ge s\}}\right]\gamma_q e^{-\gamma_q s}\diff s=\int_0^\infty\E^\alpha\left[e^{-q\rho^A(s)}\right]\gamma_q e^{-\gamma_q s}\diff s \nn\\
		&=\int_0^\infty e^{-\frac{s}{G_q^A(\alpha,\alpha)}}\gamma_q e^{-\gamma_q s}\diff s=\frac{\gamma_q\cdot G_q^A(\alpha,\alpha)}{1+\gamma_q\cdot G_q^A(\alpha,\alpha)}.
	\end{align*}
	In the fourth equality, we used the fact that the jumps of the inverse local time process $\rho^A(s)$ occur countably many times, so that the value of the integral is not affected if $\inf\{t:L^A(t)\ge s\}$ is replaced by $\inf\{t:L^A(t)>s\}=\rho^A(s)$.
\end{proof}
\section{Implementation}
In this section, we keep Assumption \ref{assump-s} valid. For the boundary conditions at the reflecting point, we refer the reader to \citet[Chapter II, Sections 1.7, 1.10]{borodina-salminen}.

\subsection {General procedure} \label{sec: procedure}
We shall describe how we obtain the decomposition formula \eqref{eq:prop-laplace}.  First, we solve $\G f=qf$ with $\G$ in \eqref{eq:diff-operator} to find $\psi_{q}(\cdot)$ and $\phi_{q}(\cdot)$ and calculate the Wronskian as well as the scale function $s(\cdot)$ in \eqref{eq:scale-speed}. By \eqref{eq:green-q} we obtain the Green function $G_q(\cdot, \cdot)$.

The procedure to obtain $G^A_q(\cdot, \cdot)$ for $q>0$ is as follows.  Let us consider $X^A$ on $\I^A=[\alpha, r)$ and denote its increasing and decreasing solutions to $\G f=qf$ as $\psi_q^A$ and $\phi_q^A$, respectively.  Recall that we can work for Proposition \ref{prop:laplace-product} with $X^A$ (rather than $\hat{X}^A$)  which is reflected at $\alpha$.  At the reflecting boundary $\alpha$, the condition is $(\psi_q^A)^+(\alpha)=0$.   Let us set $\psi_q^A(x)=a_1\psi_q(x)+a_2\phi_q(x)$ with some constants $a_1,a_2$ depending on $\alpha$, which satisfy
\begin{align*}
\begin{pmatrix}
  \psi_q^+(\alpha) & \phi_q^+(\alpha)\\
  \psi_q(\alpha) & \phi_q(\alpha)
\end{pmatrix}
\begin{pmatrix}
  a_1 \\
  a_2
\end{pmatrix}
=
\begin{pmatrix}
  0 \\
  \psi_{q}^A(\alpha)
\end{pmatrix}
.
\end{align*}
By using the Wronskian $w_q$, the solutions are $a_1(\alpha)=\frac{-\phi^+_q(\alpha) \psi_{q}^A(\alpha)}{w_q}$ and $a_2(\alpha)=\frac{\psi^+_q(\alpha) \psi_{q}^A(\alpha)}{w_q}$.  Since these solutions are unique up to constant multiplication, we choose
\begin{equation}\label{eq:a1a2}
  a_1(\alpha)=\frac{-\phi^+_q(\alpha) }{w_q}>0 \conn a_2(\alpha)=\frac{\psi^+_q(\alpha) }{w_q}>0
\end{equation}
so that $\psi_q^A(x)=\frac{-\phi^+_q(\alpha) }{w_q}\psi_q(x)+\frac{\psi^+_q(\alpha) }{w_q}  \phi_q(x)$.  This function is indeed increasing on $\I^A$. See Appendix \ref{app:increase-decrease} for details.
There is no boundary condition at $\alpha$ for $\phi^A_q$  and we set $\phi^A_q(x)=\phi_q(x)$ in $\I^A$.
From these, we compute $w^A_q$ (the Wronskian for $X^A$) also depending on $\alpha$. That is
 $w^A_q(\alpha)=(\psi_q^A)^+(x)\phi^A_q(x) -\psi_q^A(x) (\phi^A_q)^+(x)$.
Plug \eqref{eq:a1a2} in this equation and rearrange the terms to obtain
\begin{align}\label{eq:wronskian-A}
w^A_q(\alpha)&=\left[\frac{-\phi^+_q(\alpha) }{w_q}\psi_q^+(x) + \frac{\psi^+_q(\alpha) }{w_q}\phi_q^+(x)\right]\phi_q(x)-\left[\frac{-\phi^+_q(\alpha) }{w_q}\psi_q(x)+\frac{\psi^+_q(\alpha) }{w_q}\phi_q(x)\right]\phi_q^+(x) \nn\\
&=\frac{-\phi^+_q(\alpha) }{w_q}[\psi_q^+(x)\phi_q(x)-\psi_q(x)\phi_q^+(x)]=-\phi^+_q(\alpha)>0
\end{align}
by using the definition of $w_q$  in the last equality.  It is now clear why $a_1(\alpha)$ and $a_2(\alpha)$ are so chosen in \eqref{eq:a1a2}. The Green function of $X^A$ is thereby written as
\begin{equation}\label{eq:Green-A}
  G_q^A(x, y)=\frac{1}{-\phi^+_q(\alpha)}\phi_q(x)\left[\frac{-\phi^+_q(\alpha) }{w_q}\psi_q(y)+\frac{\psi^+_q(\alpha) }{w_q}  \phi_q(y)\right], \quad x\ge y.
\end{equation}
For the case $x\le y$, $\psi_q^A$ should be evaluated at $x$ and $\phi_q^A$ should be evaluated at $y$. In particular, evaluating this Green function at $(\alpha, \alpha)$, we obtain
\begin{equation}\label{eq:Green-A-at-alpha}
  G_q^A(\alpha, \alpha)=\frac{\phi_q(\alpha)}{-\phi^+_q(\alpha)}
\end{equation}
since the numerator in the square bracket of  \eqref{eq:Green-A} is $w_q$, the Wronskian of the original $X$.

Let us move on to $G_q^B(\cdot, \cdot)$ for $q>0$. Consider $X^B$ on $\I^B=(\ell, \alpha]$ and   denote its increasing and decreasing solutions to $\G f=qf$ as $\psi_q^B$ and $\phi_q^B$, respectively.  The boundary condition at the reflecting boundary $\alpha$ is $(\phi_q^B)^-(\alpha)=0$. We set  $\phi_q^B(x)=b_1\psi_q(x)+b_2\phi_q(x)$ with some constants $b_1,b_2$ depending on $\alpha$.  By proceeding similarly to the case of $X^A$, we choose
\begin{equation}\label{eq:b1b2}
  b_1(\alpha)=\frac{-\phi^-_q(\alpha) }{w_q}>0  \conn b_2(\alpha)=\frac{\psi^-_q(\alpha) }{w_q}>0
\end{equation}
so that $\phi_q^B(x)=\frac{-\phi^-_q(\alpha) }{w_q}\psi_q(x)+\frac{\psi^-_q(\alpha) }{w_q}\phi_q(x)$. This function is decreasing on $\mathcal{I}^B$ as shown in Appendix \ref{app:increase-decrease}. There is no boundary condition at $\alpha$ for $\psi_q^B$  and we set $\psi_q^B(x)=\psi_q(x)$ in $\mathcal{I}^B$. Thus, the Wronskian for $X^B$ is calculated, similarly to \eqref{eq:wronskian-A}, as
\begin{align*}
w_q^B(\alpha) = \psi_q^-(\alpha)>0.
\end{align*}
The Green function of $X^B$ is thereby written as
\begin{equation}\label{eq:Green-B}
  G_q^B(x, y)=\frac{1}{\psi^-_q(\alpha)}\psi_q(y)\left[\frac{-\phi^-_q(\alpha) }{w_q}\psi_q(x)+\frac{\psi^-_q(\alpha) }{w_q}  \phi_q(x)\right], \quad x\ge y.
\end{equation}
For the case $x\le y$, $\psi_q^B$ should be evaluated at $x$ and $\phi_q^B$ should be evaluated at $y$.  In particular, evaluating this Green function at $(\alpha, \alpha)$, we obtain
\begin{equation}\label{eq:Green-B-at-alpha}
  G_q^B(\alpha, \alpha)=\frac{\psi_q(\alpha)}{\psi^-_q(\alpha)}
\end{equation}
since the numerator in the square bracket of  \eqref{eq:Green-B} is $w_q$, the Wronskian of the original $X$.

Next we calculate $G^B_0(\cdot, \cdot)$.  We proceed similarly to the case of $G^B_q$, solving $\G f=0$ to obtain $\psi_0$ and $\phi_0$. The boundary condition at the reflecting boundary $\alpha$ is $(\phi_0^B)^-(\alpha)=0$. We set  $\phi_0^B(x)=c_1\psi_0(x)+c_2\phi_0(x)$ with some constants $c_1,c_2$ depending on $\alpha$.  Similarly to \eqref{eq:b1b2}, we have $c_1(\alpha)=\frac{-\phi^-_0(\alpha) }{w_0}$ and $c_2(\alpha)=\frac{\psi^-_0(\alpha) }{w_0}$. This implies that $\phi_0^B$ is in fact constant on $\mathcal{I}^B$ as shown in Appendix \ref{app:increase-decrease}. There is no boundary condition at $\alpha$ for $\psi_0^B$  and we set $\psi_0^B(x)=\psi_0(x)$.  Then $w_0^B(\alpha)=(\psi_0^B)^-(x)\phi^B_0(x)-\psi_0^B(x) (\phi^B_0)^-(x)=\psi_0^-(\alpha)$.  The Green function $G_0^B(x, y)$ for $x\ge y$ becomes
\begin{align}\label{eq:case1-GB}
  G_0^B(x, y)=\frac{1}{\psi_0^-(\alpha)}\psi_0(y)\left[\frac{-\phi^-_0(\alpha) }{w_0}\psi_0(x)+\frac{\psi^-_0(\alpha) }{w_0}  \phi_0(x)\right].
\end{align}
But thanks  to \eqref{lemma:psi-phi}, we can set $\phi_0=1$ and simplify further to obtain $G^B_0(x, y)=\frac{\psi_0(y)}{w_0}$.  On the other hand, the Green function of the original $X$ with $q=0$ for $x\ge y$ is
$G_0(x, y)=\frac{\psi_0(y)}{w_0}$ by \eqref{eq:green-0}. The case $x\le y$ is similar, so that we have
\begin{align*}
  G^B_0(x, y)=G_0(x, y)=\begin{cases}
    \frac{\psi_0(y)}{w_0}, \quad x\ge y\\
    \frac{\psi_0(x)}{w_0}, \quad x\le y
  \end{cases}
  =(s(x)-s(\ell))\wedge (s(y)-s(\ell)).
\end{align*}  In particular,
\begin{equation}\label{eq:G=GB}
  G^B_0(\alpha, \alpha)=G_0(\alpha, \alpha)=\frac{\psi_0(\alpha)}{w_0}=s(\alpha)-s(\ell), \quad \alpha\in \I.
\end{equation}
The quantities $a_1$, $a_2$, $b_1$, $b_2$, $w_q^A$ and $w_q^B$ all depend on $\alpha$.  But since $\alpha$ is a pre-fixed state in $\I$, we consider them constant.  For this reason, we omit the argument of these quantities for the rest of this paper. Since we have established the general procedure, we shall take specific diffusions below. For the basic characteristics of each diffusion, refer to \citet[Appendix 1]{borodina-salminen}.

\subsection{\textbf{Example:}} \label{sec:example} \emph{Brownian motion with drift}.
In this example, we consider the last passage time of the level $\alpha=0$ for a Brownian motion with drift starting at $\alpha=0$. We decompose its Laplace transform using Proposition \ref{prop:laplace-product}. Let $X$ be a Brownian motion with drift $\mu<0$ and set $\nu=-\mu>0$. The state space is $\mathcal{I}=(-\infty,+\infty)$ and both boundaries are natural.
The scale function is
\[s(x)=\frac{1}{2\nu}(e^{2\nu x}-1).\] We see that $\lim_{y\downarrow-\infty}s(y)=-\frac{1}{2\nu}>-\infty$ and $\lim_{y\uparrow +\infty}s(y)=+\infty$, so that Assumption \ref{assump-s} holds. The generator is given by $\G f(x)=\frac{1}{2}f''(x)-\nu f'(x)$.
The linearly independent solutions to $\G f=qf$ are given by
\[
\psi_q(x)=e^{(\sqrt{\nu^2+2q}+\nu)x}\conn \phi_q(x)=e^{-(\sqrt{\nu^2+2q}-\nu)x}.\]
Moreover, from \eqref{eq:right-left-derivative}, we have  $\psi_q^+(x)=(\nu+\sqrt{\nu^2+2q})e^{(\sqrt{\nu^2+2q}-\nu)x}$ and $\phi_q^+(x)=(\nu-\sqrt{\nu^2+2q})e^{-(\sqrt{\nu^2+2q}+\nu)x}$. Hence by \eqref{eq:green-q}
\[w_q=2\sqrt{\nu^2+2q}\conn G_q(x,\alpha)=\frac{1}{2\sqrt{\nu^2+2q}}e^{-(\sqrt{\nu^2+2q}-\nu)x}\cdot e^{(\sqrt{\nu^2+2q}+\nu)\alpha}\] for $x\ge \alpha$.

Let us compute $G_0(\alpha, \alpha)$ in three ways, via \eqref{eq:G0}, \eqref{eq:green-0}, and \eqref{eq:G0-explicit}, to make sure that all  lead to the same result.  First, by \eqref{eq:G0}, $\lim_{q\downarrow 0}G_q(\alpha, \alpha)=\frac{1}{2\nu}e^{2\nu \alpha}$.  To use \eqref{eq:green-0}, we solve the ODE $\G f=0$ to obtain the fundamental solutions $\psi_0(x)=e^{2\nu x}$ and $\phi_0(x)=1$, so that $w_0=2\nu$. From these results, we have
$G_0(\alpha, \alpha)=\frac{1}{w_0}\psi_0(\alpha)=\frac{1}{2\nu}e^{2\nu \alpha}$.  Finally, by \eqref{eq:G0-explicit},
 $G_0(\alpha,\alpha)=s(\alpha)-\lim_{y\downarrow-\infty}s(y)=\frac{1}{2\nu}e^{2\nu \alpha}$, which confirms the claim.  Now, we obtain from \eqref{eq:lambda-laplace} by substituting $\alpha=0$
\begin{equation}\label{eq:lambda-laplace-bm}
	\E^x\left[e^{-q\lambda_0}\right]=\frac{\nu}{\sqrt{\nu^2+2q}}e^{-(\sqrt{\nu^2+2q}-\nu)x}, \quad x\ge 0.
\end{equation}

Let us consider $X^A$ on $[0,\infty)$.   For the purpose of identifying the decomposition, we can save a lot of computations by applying \eqref{eq:Green-A-at-alpha} to obtain $G_q^A(0,0)=\frac{1}{\sqrt{\nu^2+2q}-\nu}$ with $\alpha=0$.  Next, we move on to $X^B$ on $(-\infty, 0]$.  By \eqref{eq:Green-B-at-alpha}, $G_q^B(0,0)=\frac{1}{\sqrt{\nu^2+2q}+\nu}$. On the other hand, by \eqref{eq:G=GB} we have $G^B_0(\alpha,\alpha)=G_0(\alpha,\alpha)=\frac{1}{2\nu}e^{2\nu \alpha}$ and by substitution $G^B_0(0, 0)=\frac{1}{2\nu}$.

Now we substitute our results in the decomposition formula \eqref{eq:prop-laplace} to obtain
\[
\frac{G_q^A(0, 0)}{G_q^A(0, 0)+G_q^B(0, 0)}\cdot \frac{G_q^B(0, 0)}{G_0^{B}(0,0)}
=\frac{\sqrt{\nu^2+2q}+\nu}{2\sqrt{\nu^2+2q}}\cdot \frac{2\nu}{\sqrt{\nu^2+2q}+\nu}=\frac{\nu}{\sqrt{\nu^2+2q}}
\]
which matches \eqref{eq:lambda-laplace-bm} with $x=0$. We have confirmed Proposition \ref{prop:laplace-product}.

\section{Mathematical Applications: Decomposition of Green function}\label{sec:Green}%
%\subsection{Covering the other cases}
In this section, we present some  mathematical applications of the decomposition formula \eqref{eq:prop-laplace}. For this purpose, first recall the processes $X^A$ on $\mathcal{I}^A=[\alpha,r)$ and $X^B$ on $\mathcal{I}^B=(\ell,\alpha]$.  Refer to Remark \ref{rem:three-processes} which says that, in particular, (i) both $X^A$ and $X^B$ have the same scale function and speed measure as $X$, (ii) they are reflecting at $\alpha$,  and (iii) not killed in the interior of $\mathcal{I}$, but only killed at the boundaries $r$ and $\ell$, respectively (because the original $X$ is so assumed). For the rest of the paper, $X^A$ and $X^B$ should be understood as processes on $\mathcal{I}^A$ and $\mathcal{I}^B$, respectively, satisfying (i)$\sim$(iii). Recall that superscripts $A$ and $B$ are used to denote quantities associated with these processes.
\newline\indent
If we assume, in place of Assumption \ref{assump-s}, $s(\ell)=-\infty, s(r)<+\infty$  which we refer to as Case 2 below, the construction of $X^A$ and $X^B$ (detailed in Section \ref{sec:time-changed process}) changes in an obvious way. However, the points (i)$\sim$(iii) above remain valid.  Furthermore, if we assume $s(\ell)>-\infty, s(r)<+\infty$, which is referred to as Case 3, we do not know whether $X$ is killed at $\ell$ or at $r$.

The analysis of $G^A_q(\cdot, \cdot)$ and $G^B_q(\cdot, \cdot)$ for $q>0$ does not change from the material in Section \ref{sec: procedure} in either Case 2 or Case 3.  However, we need to be careful for $q=0$.  The key observation in Section \ref{sec: procedure}  was $G^B_0(\alpha, \alpha)=G_0(\alpha, \alpha)$ in \eqref{eq:G=GB}.  The question is how we make adjustments of this fact in treating Cases 2 and 3. We shall show this in the next two remarks, respectively.  We could have used different notations of $\psi_0$, $\phi_0$, $G_0$ and $w_0$ according to the three cases.  However, we do not want to flood the exposition with notations where the reader can easily differentiate these quantities.

\begin{remark}\normalfont \label{rem:case2-G0}
$G_0(x, y)$ is represented in \eqref{eq:green-0}. For the first case $s(\ell)>-\infty, s(r)=+\infty$ (Assumption \ref{assump-s}), see the explanations preceding \eqref{eq:G0-explicit}.  In Case 2: $s(\ell)=-\infty, s(r)<+\infty$, we work similarly to the first case.  The increasing solution $\psi_0$ and decreasing solution $\phi_0$ of $\G f=0$ are uniquely determined based on the boundary conditions and satisfy
 \begin{equation*}
   \psi_0\equiv 1, \quad \phi_0(\ell+)=+\infty, \quad \phi_0(r-)=0.
 \end{equation*} Note that we can set $\psi_0=1$ rather than $\phi_0=1$.
 Then from the boundary conditions, we have
 \[
 \phi_0(x)=w_0(s(r)-s(x)), \quad x\in \I,
 \] which in turn leads to
 \begin{equation*}
   G_0(x, y) = (s(r)-s(x)) \wedge (s(r)-s(y))
 \end{equation*} since $\psi_0\equiv 1$. Note that $w_0$ is forced to be $-\phi_0^+$.

 We shall show $G_{0}^A(\alpha, \alpha)=G_0(\alpha, \alpha)$ in Case 2. For $X^A$, we denote the  increasing and decreasing solutions to $\G f=0$ as $\psi_0^A$ and $\phi_0^A$, respectively. Using the  condition at the reflecting boundary $\alpha$, we have $\psi_0^A(x)=a_1\psi_0(x)+a_2\phi_0(x)$ with $a_1=\frac{-\phi^+_0(\alpha) }{w_0}>0 \conn a_2=\frac{\psi^+_0(\alpha) }{w_0}>0$. There is no boundary condition at $\alpha$ for $\phi^A_0$ and we set $\phi_0^A(x)=\phi_0(x)$. Then $w_0^A=-\phi_0^+(\alpha)$.  The Green function $G_0^A(x, y)$ for $x\ge y$ becomes
 \begin{equation}\label{eq:case2-GA}
 G_0^A(x, y)=\frac{1}{-\phi^+_0(\alpha)}\phi_0(x)\left[\frac{-\phi^+_0(\alpha) }{w_0}\psi_0(y)+\frac{\psi^+_0(\alpha) }{w_0}  \phi_0(y)\right].
 \end{equation}  By setting $\psi_0=1$ and simplifying, we obtain $G_0^A(x, y)=\frac{\phi_0(x)}{w_0}$. On the other hand, the Green function of the original $X$ with $q=0$ for $x\ge y$ is
$G_0(x, y)=\frac{\phi_0(x)}{w_0}$ by \eqref{eq:green-0} when we set $\psi_0=1$. The case $x\le y$ is similar, so that we have
\begin{align}\label{eq:G=GA}
  G^A_0(x, y)=G_0(x, y)=\begin{cases}
    \frac{\phi_0(x)}{w_0}, \quad x\ge y\\
    \frac{\phi_0(y)}{w_0}, \quad x\le y
  \end{cases}
  =(s(r)-s(x))\wedge (s(r)-s(y)).
\end{align}  In particular, we obtain $G^A_0(\alpha, \alpha)=G_0(\alpha, \alpha)=\frac{\phi_0(\alpha)}{w_0}=s(r)-s(\alpha)$, $\alpha\in \I$.
\end{remark}

\begin{remark}\normalfont \label{rem:case3-G0}
In Case 3: $s(\ell)>-\infty, s(r)<+\infty$, we set the increasing solution $\psi_0$ and decreasing solution $\phi_0$ of $\G f=0$ as
\begin{equation}\label{eq:case3-phi-psi}
  \psi_0(y)=s(y)-s(\ell) \conn \phi_0(x)=s(r)-s(x), \quad x, y\in \I.
\end{equation} These functions satisfy the boundary condition $\psi_0(\ell_+)=\phi_0(r-)=0$.  Then we compute the constant $w_0$ and observe
\begin{equation}\label{eq:case3-w0}
  w_0=s(r)-s(\ell)>0 \conn \psi_0(x)+\phi_0(x)=w_0, \quad \forall x\in \I.
\end{equation}
Plugging these quantities into \eqref{eq:green-0}, we have
\begin{align*}
	G_0(x,y)=\begin{cases}
\frac{(s(y)-s(\ell))(s(r)-s(x))}{s(r)-s(\ell)}, \quad x\ge y, \\
		\frac{(s(x)-s(\ell))(s(r)-s(y))}{s(r)-s(\ell)}, \quad x\le y.
			\end{cases}
\end{align*}
It is easy to see that we can retrieve the form of $G_0(x, y)$ in Cases 1 and 2 by using their respective conditions at $\ell$ and $r$.

Let us consider processes $X^A$ and $X^B$.
Using the same method as in Remark \ref{rem:case2-G0}, it is easy to see that $G_0^A(x, y)$ for $x\ge y$ in this case becomes
\[G_0^A(x, y)=\frac{1}{-\phi^+_0(\alpha)}\phi_0(x)\left[\frac{-\phi^+_0(\alpha) }{w_0}\psi_0(y)+\frac{\psi^+_0(\alpha) }{w_0}  \phi_0(y)\right],\]
which is the same as \eqref{eq:case2-GA} except for $w_0=s(r)-s(\ell)$.  Noting that $\psi_0=1$  and $\phi_0^+(\alpha)=-1$, we simplify it with \eqref{eq:case3-w0} to find  $G_0^A(x, y)=\phi_0(x)$.  The case $x\le y$ is similar, so that we have
 \begin{align}\label{eq:case3-GA_0}
  G^A_0(x, y)=\begin{cases}
   \phi_0(x), \quad x\ge y \\
    \phi_0(y), \quad x\le y %\nn
  \end{cases}
  =(s(r)-s(x))\wedge (s(r)-s(y)).
  \end{align}
On the other hand, following the procedure in Section \ref{sec: procedure}, $G_0^B(x, y)$ for $x\ge y$ is
\begin{align*}
  G_0^B(x, y)=\frac{1}{\psi_0^-(\alpha)}\psi_0(y)\left[\frac{-\phi^-_0(\alpha) }{w_0}\psi_0(x)+\frac{\psi^-_0(\alpha) }{w_0}  \phi_0(x)\right],
\end{align*}
which is the same as \eqref{eq:case1-GB} except for $w_0=s(r)-s(\ell)$. Again noting that $\psi_0^-(\alpha)=1$ and $\phi_0^-(\alpha)=-1$, we simplify it with \eqref{eq:case3-w0} to find  $G_0^B(x, y)=\psi_0(y)$.  The case $x\le y$ is similar, so that we have
\begin{align}\label{eq:case3-GB_0}
  G^B_0(x, y)=\begin{cases}
    \psi_0(y), \quad x\ge y\\
    \psi_0(x), \quad x\le y
  \end{cases}
  =(s(x)-s(\ell))\wedge (s(y)-s(\ell)).
\end{align}
Neither $G^A_0(x, y)$ nor $G^B_0(x, y)$ is equal to $G_0(x, y)$.  This differs from the other two cases.  However, we have the following observations: by comparison we write $G^A_0(x, y)=\frac{G_0(x, y)w_0}{\psi_0(y)}$ and $G^B_0(x, y)=\frac{G_0(x, y)w_0}{\phi_0(x)}$ for $x\ge y$.  Then by \eqref{eq:case3-w0}, we have
\begin{equation}\label{eq:GA+GB=G}
  \frac{1}{G^A_0(\alpha, \alpha)}+\frac{1}{G^B_0(\alpha, \alpha)}=\frac{1}{G_0(\alpha, \alpha)},
\end{equation} which is the key to prove the decomposition formula for Case 3 and can be viewed as a decomposition formula of the Green function when $q=0$.

\end{remark}

\begin{proposition}\label{prop:all}
	In all three cases of
	(1) $s(\ell)>-\infty, s(r)=+\infty$,
	(2) $s(\ell)=-\infty, s(r)<+\infty$, and
	(3) $s(\ell)>-\infty, s(r)<+\infty$, we have for $q>0$
	\begin{equation}\label{eq:all-last-time}
		\E^\alpha[e^{-q\lambda_\alpha}]=\frac{G_q^A(\alpha, \alpha)}{G^A_q(\alpha, \alpha)+G^B_q(\alpha, \alpha)}\cdot \frac{G^B_q(\alpha, \alpha)}{G_0(\alpha, \alpha)}
	\end{equation}
	where $G^A_\cdot(\cdot, \cdot)$ and $G^B_\cdot(\cdot, \cdot)$ are the Green functions of $X^A$  and $X^B$, respectively, and $G_0(\cdot, \cdot)$ is the Green function of the original $X$ defined in \eqref{eq:G0}: specifically,
\begin{align}\label{eq:G0-all}
  G_0(\alpha, \alpha)&=\lim_{c\downarrow \ell, d\uparrow r}\frac{(s(\alpha)-s(c))(s(d)-s(\alpha))}{s(d)-s(c)}=
  \begin{cases}
   s(\alpha)-s(\ell), \quad \text{Case 1},\\
   s(r)-s(\alpha), \quad \text{Case 2},\\
   \frac{(s(\alpha)-s(\ell))(s(r)-s(\alpha))}{s(r)-s(\ell)}, \quad\text{Case 3}.
  \end{cases}
\end{align}
\end{proposition}
\begin{proof}
	\textbf{Case 1}: $s(\ell)>-\infty, s(r)=+\infty$.
	This is the case of Assumption \ref{assump-s}.  We just note that $G_0(\alpha, \alpha)=G^B_0(\alpha, \alpha)=s(\alpha)-s(\ell)$ for $\alpha\in \I$  as in \eqref{eq:G=GB} to see that  \eqref{eq:all-last-time} is identical to \eqref{eq:prop-laplace} in Proposition \ref{prop:laplace-product}.

	\noindent \textbf{Case 2}: $s(\ell)=-\infty, s(r)<+\infty$. The  same proof  as for Proposition \ref{prop:laplace-product} leads to the decomposition formula as \eqref{eq:prop-laplace} with the roles of $G_0^A(\alpha, \alpha)$ and $G_0^B(\alpha, \alpha)$ interchanged. We obtain \eqref{eq:all-last-time} by noting that $G_0(\alpha, \alpha)=G^A_0(\alpha, \alpha)=s(r)-s(\alpha)$ for $\alpha\in \I$  as in \eqref{eq:G=GA} in Remark \ref{rem:case2-G0}.

\noindent \textbf{Case 3}: $s(\ell)>-\infty, s(r)<+\infty$. In this case, both events $\{\lim_{t\to\xi}X_t=\ell\}$ and $\{\lim_{t\to\xi}X_t=r\}$ occur with positive probability because both boundaries are attracting. We have
\begin{equation}\label{eq:hitting-prob}
	\p^\alpha(\lim_{t\to\xi}X_t=\ell)=1-\p^\alpha(\lim_{t\to\xi}X_t=r)=\frac{s(r)-s(\alpha)}{s(r)-s(\ell)}
\end{equation}
as shown in Proposition 5.22 in \citet[Chapter 5]{karatzas}. Note that
\begin{align}\label{eq:laplace-cond}
	\E^\alpha[e^{-q\lambda_\alpha }]
	=\E^\alpha[e^{-q\lambda_\alpha}\mid \lim_{t\to\xi}X_t=\ell ]\cdot \p^\alpha(\lim_{t\to\xi}X_t=\ell)+\E^\alpha[e^{-q\lambda_\alpha}\mid \lim_{t\to\xi}X_t=r]\cdot \p^\alpha(\lim_{t\to\xi}X_t=r).
\end{align}
We use Doob's $h$-transform to proceed with the proof. Refer to Appendix \ref{app:elements} for details regarding $h$-transform. The idea is to force the process $X$ to the left (resp. right) boundary so that we can use the results of Case 1 (resp. Case 2). It is well known that the $h$-transform of $X$ with the excessive function $\phi_0$ (resp. $\psi_0$) in \eqref{eq:case3-phi-psi} is identical in law to the original $X$ conditioned on $\{\lim_{t\to \xi}X_t=\ell\}$ (resp. $\{\lim_{t\to \xi}X_t=r\}$). See the references mentioned in Appendix \ref{app:elements}.

Let us first consider the $h$-transform of $X$ with $\phi_0(x)=s(r)-s(x)$. We denote this transformed diffusion as $X^*$ and use $*$ to indicate quantities associated with it. The scale function becomes $s^*(x)=\frac{1}{\phi_0(x)}$ and we see that $X^*$ satisfies Assumption \ref{assump-s}. The Green function of $X^*$ is given by
\begin{equation*}
	G_q^*(x,y)=\frac{G_q(x,y)}{\phi_0(x)\phi_0(y)}, \quad x,y\in\mathcal{I}
\end{equation*}
from which we see that $\psi_q^*=\frac{\psi_q}{\phi_0}$, $\phi_q^*=\frac{\phi_q}{\phi_0}$, and $w_q^*=w_q$ by direct computation. We can also check that $\psi_q^*$  is increasing and $\phi_q^*$ is decreasing. See Appendix \ref{app:case3-analysis} for a proof. Note that for $X^*$, the right and left derivatives with respect to the scale function should be calculated based on $s^*$.

Let us consider $(X^*)^A$ on $[\alpha, r)$ and $(X^*)^B$ on $(\ell, \alpha]$ which have the same scale function and speed measure as $X^*$ and are reflecting at $\alpha$. Recall that killing occurs only at boundaries $\ell$ and $r$. Using \eqref{eq:Green-A-at-alpha} and \eqref{eq:Green-B-at-alpha}, we have
\begin{equation*}
	(G_q^*)^A(\alpha,\alpha)=\frac{\phi_q^*(\alpha)}{-(\phi_q^*)^+(\alpha)}=\frac{G_q^A(\alpha,\alpha)}{\phi_0^2(\alpha)-\phi_0(\alpha)G_q^A(\alpha,\alpha)},
\end{equation*}

\begin{equation*}
	(G_q^*)^B(\alpha,\alpha)=\frac{\psi_q^*(\alpha)}{(\psi_q^*)^-(\alpha)}=\frac{G_q^B(\alpha,\alpha)}{\phi_0^2(\alpha)+\phi_0(\alpha)G_q^B(\alpha,\alpha)}.
\end{equation*}
Furthermore, due to \eqref{eq:case3-phi-psi} and \eqref{eq:case3-w0}, \eqref{eq:G=GB} results in
\begin{equation*}
	(G_0^*)^B(\alpha,\alpha)=s^*(\alpha)-s^*(\ell)=\frac{1}{\phi_0(\alpha)}-\frac{1}{\phi_0(\ell)}=\frac{\psi_0(\alpha)}{\phi_0(\alpha)w_0}.
\end{equation*}
Then, from the decomposition formula \eqref{eq:prop-laplace}, \eqref{eq:case3-GB_0}, and \eqref{eq:hitting-prob}, we obtain
\begin{align*}
	\E^\alpha[e^{-q\lambda_\alpha}\mid \lim_{t\to\xi}X_t=\ell]\cdot \p^\alpha(\lim_{t\to\xi}X_t=\ell)&=\frac{(G_q^*)^A(\alpha, \alpha)}{(G^*_q)^A(\alpha, \alpha)+(G^*_q)^B(\alpha, \alpha)}\cdot \frac{(G^*_q)^B(\alpha, \alpha)}{(G^*_0)^B(\alpha, \alpha)}\cdot \frac{\phi_0(\alpha)}{w_0}\\
	&=\frac{G_q^A(\alpha,\alpha)\cdot G_q^B(\alpha,\alpha)}{G_q^A(\alpha,\alpha)+G_q^B(\alpha,\alpha)}\cdot \frac{1}{G_0^B(\alpha,\alpha)}.
\end{align*}

Next, we consider the $h$-transform of $X$ with $\psi_0(x)=s(x)-s(\ell)$. We denote this transformed diffusion as $\tilde{X}$ and use $\sim$ to indicate quantities associated with it. The scale function becomes $\tilde{s}(x)=-\frac{1}{\psi_0(x)}$ and we see that $\tilde{X}$ belongs to Case 2. The Green function is given by
\begin{equation*}
	\tilde{G}_q(x,y)=\frac{G_q(x,y)}{\psi_0(x)\psi_0(y)}, \quad x,y\in\mathcal{I}
\end{equation*}
from which we obtain $\tilde{\psi}_q=\frac{\psi_q}{\psi_0}$, $\tilde{\phi}_q=\frac{\phi_q}{\psi_0}$, and $\tilde{w}_q=w_q$ by direct computation. We can also check that $\tilde{\psi}_q$  is increasing and $\tilde{\phi}_q$ is decreasing. See Appendix \ref{app:case3-analysis} for a proof.

Let us consider $\tilde{X}^A$ on $[\alpha, r)$ and $\tilde{X}^B$ on $(\ell, \alpha]$ which have the same scale function and speed measure as $\tilde{X}$ and are reflecting at $\alpha$, killing occurring only at $\ell$ and $r$. Using \eqref{eq:Green-A-at-alpha} and \eqref{eq:Green-B-at-alpha}, we have
\begin{equation*}
	\tilde{G}_q^A(\alpha,\alpha)=\frac{\tilde{\phi}_q(\alpha)}{-\tilde{\phi}_q^+(\alpha)}=\frac{G_q^A(\alpha,\alpha)}{\psi_0^2(\alpha)+\psi_0(\alpha)G_q^A(\alpha,\alpha)},
\end{equation*}
\begin{equation*}
	\tilde{G}_q^B(\alpha,\alpha)=\frac{\tilde{\psi}_q(\alpha)}{\tilde{\psi}_q^-(\alpha)}=\frac{G_q^B(\alpha,\alpha)}{\psi_0^2(\alpha)-\psi_0(\alpha)G_q^B(\alpha,\alpha)}.
\end{equation*}
Furthermore, due to \eqref{eq:case3-phi-psi} and \eqref{eq:case3-w0}, \eqref{eq:G=GA} results in
\begin{equation*}
	\tilde{G}_0^A(\alpha,\alpha)=\tilde{s}(r)-\tilde{s}(\alpha)=-\frac{1}{\psi_0(r)}+\frac{1}{\psi_0(\alpha)}=\frac{\phi_0(\alpha)}{\psi_0(\alpha)w_0}.
\end{equation*}
Then, the decomposition formula for Case 2, \eqref{eq:case3-GA_0}, and \eqref{eq:hitting-prob} provide
\begin{align*}
	\E^\alpha[e^{-q\lambda_\alpha}\mid \lim_{t\to\xi}X_t=r]\cdot \p^\alpha(\lim_{t\to\xi}X_t=r)&=\frac{\tilde{G}_q^A(\alpha, \alpha)}{\tilde{G}_q^A(\alpha, \alpha)+\tilde{G}_q^B(\alpha, \alpha)}\cdot \frac{\tilde{G}_q^B(\alpha, \alpha)}{\tilde{G}_0^A(\alpha, \alpha)}\cdot \frac{\psi_0(\alpha)}{w_0}\\
	&=\frac{G_q^A(\alpha,\alpha)\cdot G_q^B(\alpha,\alpha)}{G_q^A(\alpha,\alpha)+G_q^B(\alpha,\alpha)}\cdot \frac{1}{G_0^A(\alpha,\alpha)}.
\end{align*}
We have derived the expressions for both elements in \eqref{eq:laplace-cond}. Finally, \eqref{eq:all-last-time} holds due to \eqref{eq:GA+GB=G}.
\end{proof}

\renewcommand{\arraystretch}{1.25}
\begin{table}[h]
\centering
\caption{Summary}
\begin{tabular}{|c |c|c|c|}
  \hline
 
 &   Case 1 & Case 2 & Case 3 \\
 \hline
    $w_0$ (constant) & $\psi_0^-$ & $-\phi_0^+$ & $s(r)-s(\ell)$ \\
  $\psi_0(x)$ & $w_0(s(x)-s(\ell))$ & $1$ & $s(x)-s(\ell)$ \\
  $\phi_0(x)$ & 1 & $w_0(s(r)-s(x))$ & $s(r)-s(x)$ \\
  $G_0(x, y), \; x\ge y$ & $s(y)-s(\ell)$ & $s(r)-s(x)$ & $\frac{(s(y)-s(\ell))(s(r)-s(x))}{s(r)-s(\ell)}$ \\
  \hdashline
  $G_q^A(\alpha, \alpha)$ &$\frac{\phi_q(\alpha)}{-\phi_q^+(\alpha)}$  & $\frac{\phi_q(\alpha)}{-\phi_q^+(\alpha)}$ & $\frac{\phi_q(\alpha)}{-\phi_q^+(\alpha)}$  \\
  $G_q^B(\alpha, \alpha)$ & $\frac{\psi_q(\alpha)}{\psi_q^-(\alpha)}$  & $\frac{\psi_q(\alpha)}{\psi_q^-(\alpha)}$ & $\frac{\psi_q(\alpha)}{\psi_q^-(\alpha)}$ \\
  $G^A_0(x, y), \; x\ge y$ & $+\infty$ & $s(r)-s(x)$ & $s(r)-s(x)$ \\
  $G^B_0(x, y), \; x\ge y$ & $s(y)-s(\ell)$ & $+\infty$ & $s(y)-s(\ell)$\\
  \hline
\end{tabular}
\label{summary}
\end{table}

\renewcommand{\arraystretch}{1}

In Table \ref{summary}, we collect some results from Sections \ref{sec:setup} and \ref{sec: procedure} and  Remarks \ref{rem:case2-G0} and \ref{rem:case3-G0}. The first step is to obtain $\psi_0$ and $\phi_0$ from the solutions of $\G f =0$ as well as the scale function $s$. Then, the relationship of $\psi_0$, $\phi_0$, $w_0$, $s$ and $G_0$ is as described in the upper half of Table \ref{summary}. To obtain the decomposition, after computing $\psi_q$ and $\phi_q$ from $\G f=qf$, refer to the bottom half. Note that for rows 4, 7, and 8, $s(x)$ should be replaced by $s(y)$ (and vice versa) in case $x\le y$: except for $G^A_0(x, y)=+\infty$ in Case 1 and $G^B_0(x, y)=+\infty$ in Case 2, irrespective of the order of $x$ and $y$.

\bigskip
 
Let us now extend in another direction by starting $X$ at $x\neq \alpha$.  By using the shift operator,
\[
\lambda_\alpha= (H_\alpha + \lambda_\alpha\circ \theta_{H_\alpha})\cdot \1_{\{H_\alpha<+\infty\}}.
\]  Recall that $H_\alpha=+\infty$ implies $\lambda_\alpha=0$ and vice versa. Hence by the strong Markov property at $H_\alpha$,
\begin{align}\label{eq:x-alpha}
	\E^x[e^{-q\lambda_\alpha}]&=1\cdot \p^x(H_\alpha=+\infty)+\E^x[e^{-q(H_\alpha+\lambda_\alpha\circ \theta_{H_\alpha})}\cdot \1_{\{H_\alpha<+\infty\}}]\nn\\
	&=\p^x(H_\alpha=+\infty)+\E^x[\1_{\{H_\alpha<+\infty\}}\cdot e^{-q H_\alpha}\cdot \E^x[e^{-q\lambda_\alpha\circ \theta_{H_\alpha}}\mid \F_{H_\alpha}]]\nn\\
	&=\p^x(H_\alpha=+\infty)+\E^x[\1_{\{H_\alpha<+\infty\}}\cdot e^{-q H_\alpha}\cdot \E^\alpha[e^{-q\lambda_\alpha}]]\nn\\
	&=\p^x(H_\alpha=+\infty)+\E^x[e^{-qH_\alpha}]\cdot \E^\alpha[e^{-q\lambda_\alpha}], \qquad x\in\mathcal{I}.
\end{align}

\begin{theorem}\label{theorem}
	In all three cases of
	(1) $s(\ell)>-\infty, s(r)=+\infty$,
	(2) $s(\ell)=-\infty, s(r)<+\infty$, and
	(3) $s(\ell)>-\infty, s(r)<+\infty$, we have for any $x, \alpha\in \mathcal{I}$ and $q>0$
	\begin{align}\label{eq:x-alpha-all-last-time}
		\E^x[e^{-q\lambda_\alpha}]=
		\begin{cases}
		\lim\limits_{d\uparrow r} \frac{s(x)-s(\alpha)}{s(d)-s(\alpha)}+\frac{G_q^A(x,\alpha)} {G^A_q(\alpha, \alpha)+G^B_q(\alpha, \alpha)}\cdot \frac{G^B_q(\alpha, \alpha)}{G_0(\alpha, \alpha)}, \quad x\ge  \alpha,\\	
\lim\limits_{c\downarrow \ell}\frac{s(\alpha)-s(x)}{s(\alpha)-s(c)}+\frac{G_q^B(x,\alpha)}{G^A_q(\alpha, \alpha)+G^B_q(\alpha, \alpha)}\cdot \frac{G_q^A(\alpha, \alpha)}{G_0(\alpha, \alpha)}, \quad x\le \alpha,
		\end{cases}
	\end{align}
	where $G^A_\cdot(\cdot, \cdot)$ and $G^B_\cdot(\cdot, \cdot)$ are the Green functions of $X^A$  and $X^B$, respectively, and $G_0(\cdot, \cdot)$ is the Green function of the original $X$ defined in \eqref{eq:G0} and described in \eqref{eq:G0-all}.
	
	In all three cases, for any $x, \alpha\in \mathcal{I}$ and $q>0$,
	\begin{align}\label{eq:x-alpha-Green-decomp}
		G_q(x, \alpha)=
		\begin{cases}
\frac{G_q^A(x, \alpha)\cdot G^B_q(\alpha, \alpha)}{G^A_q(\alpha, \alpha)+G^B_q(\alpha, \alpha)}, \quad x\ge  \alpha,\\
\frac{G_q^A(\alpha, \alpha)\cdot G^B_q(x, \alpha)}{G^A_q(\alpha, \alpha)+G^B_q(\alpha, \alpha)}, \quad x\le \alpha.
					\end{cases}
	\end{align}
\end{theorem}

\begin{proof}
First, note that
\[\p^x(H_\alpha=+\infty)=\begin{cases}
	\p^x(H_r\le H_\alpha)=\lim\limits_{d\uparrow r} \frac{s(x)-s(\alpha)}{s(d)-s(\alpha)}, \quad x\ge\alpha,\\
	\p^x(H_\ell\le H_\alpha)=\lim\limits_{c\downarrow \ell}\frac{s(\alpha)-s(x)}{s(\alpha)-s(c)}, \quad x\le\alpha.
\end{cases}\]
Take the case of $x\ge \alpha$. In this case, $\E^x[e^{-qH_\alpha}]=\frac{\phi_q(x)}{\phi_q(\alpha)}$ (see \eqref{eq:hitting-time-laplace}). Note also that from \eqref{eq:green-q},
\[\frac{\phi_q(x)}{\phi_q(\alpha)}G_q(\alpha, \alpha)=\frac{\phi_q(x)}{\phi_q(\alpha)}\frac{\phi_q(\alpha)\psi_q(\alpha)}{w_q}=G_q(x, \alpha).\]  Based on Section \ref{sec: procedure}, we know that $\phi_q^A(x)=\phi_q(x)$ on $\I^A=[\alpha,r)$. Then, plugging \eqref{eq:all-last-time} into the right-hand side of \eqref{eq:x-alpha} yields the first equation of \eqref{eq:x-alpha-all-last-time}.
The case of $x\le \alpha$ is similar. The difference is that we have $\psi_q^B(x)=\psi_q(x)$ on $\I^B=(\ell,\alpha]$ as was shown in Section \ref{sec: procedure} based on boundary conditions.

Finally, for the decomposition of the Green function of $X$, we observe that
\begin{align*}
\E^x[e^{-q\lambda_\alpha}]=\p^x(\lambda_\alpha=0)+\E^x[e^{-q\lambda_\alpha}\cdot \1_{\{\lambda_\alpha>0\}}]&=\p^x(H_\alpha=+\infty)+\int_0^\infty e^{-qt} \frac{p(t;x,\alpha)}{G_0(\alpha,\alpha)}\diff t\\
&=\p^x(H_\alpha=+\infty)+\frac{G_q(x,\alpha)}{G_0(\alpha,\alpha)}.
\end{align*}
See \eqref{eq:lambda-laplace}.
From \eqref{eq:x-alpha} and \eqref{eq:x-alpha-all-last-time}, this equation implies that
\[\frac{G_q(x,\alpha)}{G_0(\alpha,\alpha)}=\begin{cases}
	\frac{G_q^A(x,\alpha)} {G^A_q(\alpha, \alpha)+G^B_q(\alpha, \alpha)}\cdot \frac{G^B_q(\alpha, \alpha)}{G_0(\alpha, \alpha)}, \quad x\ge  \alpha,\\	
	\frac{G_q^B(x,\alpha)}{G^A_q(\alpha, \alpha)+G^B_q(\alpha, \alpha)}\cdot \frac{G_q^A(\alpha, \alpha)}{G_0(\alpha, \alpha)}, \quad x\le \alpha,
\end{cases}\]
from which we obtain \eqref{eq:x-alpha-Green-decomp}.
\end{proof}

\subsection{\textbf{Example:} Ornstein-Uhlenbeck process}\label{sec:OU}
Let us illustrate the Green function's decomposition \eqref{eq:x-alpha-Green-decomp} for Case 3: $s(\ell)>-\infty, s(r)<+\infty$ by using a more complicated example of the Ornstein-Uhlenbeck process. Let $X$ follow the dynamics $\diff X_t=-\kappa X_t\diff t+\diff W_t$ on $\mathcal{I}=(-\infty,+\infty)$ with constant $\kappa<0$ and a standard Brownian motion $W$. The scale function of $X$ is given by $s(x)=\int_0^xe^{\kappa y^2}\diff y$. We see that both boundaries $-\infty$ and $+\infty$ are attracting, so that $X$ belongs to Case 3.

The increasing and decreasing linearly independent solutions of $\G f=-\kappa xf'+\frac{1}{2}f''=qf$ are given by \[\psi_q(x)=e^{-{|\kappa|\frac{x^2}{2}}}D_{-\left(\frac{q}{|\kappa|}+1\right)}(-x\sqrt{2|\kappa|}) \conn
\phi_q(x)=e^{-{|\kappa|\frac{x^2}{2}}}D_{-\left(\frac{q}{|\kappa|}+1\right)}(x\sqrt{2|\kappa|}).\]
Here $D_{-\nu} (x)$ and $D_{-\nu} (-x)$ are parabolic cylinder functions which represent linearly independent solutions of the differential equation $f''(x)-\left(\frac{x^2}{4}+\frac{2\nu-1}{2}\right)f(x)=0$ for $x\in\R$ (see \citet[Appendix 2.9]{borodina-salminen}). We could proceed without computing the constants $a_1, a_2, b_1, b_2, w_q^A$, and $w_q^B$ below as in Section \ref{sec:example}. However, since the solutions involve special functions, we shall record them.

Refer to the procedure in Section \ref{sec: procedure}. First, consider $X^A$ on $\mathcal{I}^A$. We have $\psi_q^A(x)=a_1\psi_q(x)+a_2\phi_q(x)$ with $a_1=\frac{\Gamma\left(\frac{q}{|\kappa|}+1\right)}{\sqrt{2\pi}}e^{{|\kappa|\frac{\alpha^2}{2}}}D_{-\frac{q}{|\kappa|}}(\alpha\sqrt{2|\kappa|})$ and $a_2=\frac{\Gamma\left(\frac{q}{|\kappa|}+1\right)}{\sqrt{2\pi}}e^{{|\kappa|\frac{\alpha^2}{2}}}D_{-\frac{q}{|\kappa|}}(-\alpha\sqrt{2|\kappa|})$, while $\phi_q^A(x)=\phi_q(x)$ on $\mathcal{I}^A$. The Wronskian is given by
$w_q^A=\sqrt{2|\kappa|}e^{{|\kappa|\frac{\alpha^2}{2}}}D_{-\frac{q}{|\kappa|}}(\alpha\sqrt{2|\kappa|})$.

Next, consider $X^B$ on $\mathcal{I}^B$. We have $\phi_q^B(x)=b_1\psi_q(x)+b_2\phi_q(x)$ with $b_1=\frac{\Gamma\left(\frac{q}{|\kappa|}+1\right)}{\sqrt{2\pi}}e^{{|\kappa|\frac{\alpha^2}{2}}}D_{-\frac{q}{|\kappa|}}(\alpha\sqrt{2|\kappa|})$ and $b_2=\frac{\Gamma\left(\frac{q}{|\kappa|}+1\right)}{\sqrt{2\pi}}e^{{|\kappa|\frac{\alpha^2}{2}}}D_{-\frac{q}{|\kappa|}}(-\alpha\sqrt{2|\kappa|})$, while $\psi_q^B(x)=\psi_q(x)$ on $\mathcal{I}^B$. The Wronskian is given by $w_q^B=\sqrt{2|\kappa|}e^{{|\kappa|\frac{\alpha^2}{2}}}D_{-\frac{q}{|\kappa|}}(-\alpha\sqrt{2|\kappa|})$.
Feed the above facts regarding  $X^A$ and $X^B$ into the right-hand side of the decomposition formula \eqref{eq:x-alpha-Green-decomp} in Theorem \ref{theorem} to obtain
\begin{align*}
G_q(x,\alpha)=\begin{cases}
\frac{e^{-|k|\frac{\alpha^2}{2}}D_{-\left(\frac{q}{|\kappa|}+1\right)}(-\alpha\sqrt{2|\kappa|})\times e^{-|k|\frac{x^2}{2}}D_{-\left(\frac{q}{|\kappa|}+1\right)}(x\sqrt{2|\kappa|})}{\frac{2\sqrt{\pi|\kappa|}}{\Gamma\left(\frac{q}{|\kappa|}+1\right)}}, \quad x\ge \alpha,\\
\frac{e^{-|k|\frac{\alpha^2}{2}}D_{-\left(\frac{q}{|\kappa|}+1\right)}(\alpha\sqrt{2|\kappa|})\times e^{-|k|\frac{x^2}{2}}D_{-\left(\frac{q}{|\kappa|}+1\right)}(-x\sqrt{2|\kappa|})}{\frac{2\sqrt{\pi|\kappa|}}{\Gamma\left(\frac{q}{|\kappa|}+1\right)}}, \quad x\le\alpha,\\
\end{cases}
\end{align*}
which is confirmed to be the same as the Green function in \citet[Appendix 1.24]{borodina-salminen}.

\section{Practical Applications}\label{sec:appl}
\subsection{Parameter-switching diffusion process}\label{sec:parameter-switch}

Let  $X$  be a diffusion on $\mathcal{I}=(\ell,r)\subset\R$ whose parameters are different above and below some fixed level $\alpha\in \mathcal{I}$.  This type of diffusion is useful in treating real-life problems.  For example, \citet[Section 6.5]{karatzas} considers a stochastic control problem.

Let the infinitesimal drift and diffusion parameters of $X$ be $\mu(\cdot)$ and $\sigma(\cdot)$, respectively, such that
\begin{align*}
	\mu(x)&=\mu^B(x)\1_{(\ell,\alpha)}(x)+\mu^A(x)\1_{[\alpha,r)}(x),\\ \nn
	\sigma(x)&=\sigma^B(x)\1_{(\ell,\alpha)}(x)+\sigma^A(x)\1_{[\alpha,r)}(x).
\end{align*}
The parameters $\mu^A(\cdot)$, $\sigma^A(\cdot)$, $\mu^B(\cdot)$, $\sigma^B(\cdot)$ are such that $X$ is transient, belonging to either Case 1, 2 or 3. The decomposition method of the last passage time  is convenient when dealing with such processes. Proposition 3 allows us to bypass (often) hard calculations related to $X$ (with switching parameters) and to reduce the object to two processes $X^A$ and $X^B$ with no switching parameters. Moreover, one may find \eqref{eq:x-alpha-Green-decomp} useful in identifying the Green function of parameter-switching diffusion. We point out that there is no general established method for obtaining the Green function of such diffusions explicitly in the previous literature.

Note that in the case of switching parameters, we decompose the process into two processes and deal with them separately. Therefore, if the conditions imposed on parameters (see Section \ref{sec:setup}) hold for each diffusion separately, the results apply without any need of modification.

We illustrate the decomposition scheme by the example of a Brownian motion $X$ with two-valued drift on $\R$:
\begin{align}\label{eq:X-switching}
	\diff X_t=\mu(X_t)\diff t+\diff W_t	\quad \text{with}\quad 
	\mu(X_t)=\begin{cases}
		\mu^A, \quad X_t\ge 0,\\
		\mu^B, \quad X_t< 0, 
	\end{cases}
\end{align}
with constants $\mu^A<0$, $\mu^B<0$ and $W$ a standard one-dimensional Brownian motion. Note that $X$ belongs to Case 1. Based on \eqref{eq:X-time-change}, we transform $X$ into $X^A$ and $X^B$ which are two Brownian motions with constant drifts reflecting at $0$ from above and below, respectively. Proposition \ref{prop:laplace-product} allows us to treat these two processes separately.

Consider $X^A$ on $[0,+\infty)$ with drift $\mu^A$ and $X^B$ on $(-\infty,0]$ with drift $\mu^B$. From Section \ref{sec:example}, we obtain

\[G_q^A(0,0)=\frac{1}{\sqrt{(\mu^A)^2+2q}+\mu^A}, \quad 	G_q^B(0,0)=\frac{1}{\sqrt{(\mu^B)^2+2q}-\mu^B}, \quad G^B_0(0,0)=-\frac{1}{2\mu^B}. \]
Then, Proposition \ref{prop:laplace-product} yields
\begin{align}\label{eq:lambda-laplace-switching}
	\E^0[e^{-q\lambda_0}]&=\frac{G_q^A(0, 0)}{G^A_q(0, 0)+G^B_q(0,0)}\cdot \frac{G^B_q(0, 0)}{G^B_0(0, 0)}=\frac{-2\mu^B}{\mu^A-\mu^B+\sqrt{(\mu^A)^2+2q}+\sqrt{(\mu^B)^2+2q}}.
\end{align}
This provides the Laplace transform of the last passage time for $X$ in \eqref{eq:X-switching} which has the switching parameters.
This result can be confirmed by \cite{benes1980} where they derive the Laplace transform of the transition density function (with respect to the Lebesgue measure) of $X$ in \eqref{eq:X-switching}. For this, they use the symmetry of the Brownian motion, the forward Kolmogorov equation satisfied by the transition density function, and a linear system of equations based on various conditions satisfied by the density's Laplace transform. Note that we have managed to obtain \eqref{eq:lambda-laplace-switching} without computing this transition density.

Moreover, by Proposition \ref{prop:all} the Green function of the parameter-switching diffusion $X$ is found to be
\[
G_q(0,0)=\frac{1}{\mu^A-\mu^B+\sqrt{(\mu^A)^2+2q}+\sqrt{(\mu^B)^2+2q}}.
\]
Again, there is no need to integrate the transition density function for this result.  Since the Green function appears in various contexts of mathematical problems, it is convenient  to have it available in its explicit form for a compounded diffusion like this. The method presented in this section is applicable to all transient diffusions.

\subsection{Leverage effect}\label{sec:financial-appl}
Consider a geometric Brownian motion $X$ on $\mathcal{I}=(0,+\infty)$: $\diff X_t=\mu(X_t) X_t\diff t+\sigma(X_t) X_t\diff W_t$ with switching parameters
\begin{align*}
		\mu(x)=\mu^B\1_{(-\infty,\alpha)}(x)+\mu^A\1_{[\alpha,\infty)}(x) \conn
	\sigma(x)=\sigma^B\1_{(-\infty,\alpha)}(x)+\sigma^A\1_{[\alpha,\infty)}(x),
\end{align*}
where $W$ denotes a standard Brownian motion. The parameters are such that $\nu^i:=\frac{\mu^i}{(\sigma^i)^2}-\frac{1}{2}< 0$ for $i=A, B$ which means that Assumption \ref{assump-s} is satisfied.

Let us set $X_0=\alpha$. We may interpret $X$ as a stock market price. When its value decreases beyond $\alpha$, the ratio of a firm's debt over its equity market value (leverage ratio) increases, given the value of debt is unchanged. In financial markets, it is observed that such increase in leverage is associated with higher stock price volatility. See, for example, \citet{bae2007}. Then, the assumption $\sigma^B>\sigma^A$ would capture such leverage effect.

Let us decompose $X$ into $X^A$ on $[\alpha,+\infty)$ and $X^B$ on $(0,\alpha]$, both reflecting at $\alpha$, and use the procedure in Section \ref{sec: procedure}.
For $i=A,B$, the increasing and decreasing linearly independent solutions to $\G^i f:=\mu^ixf'(x)+\frac{1}{2}(\sigma^i)^2x^2f''(x)=qf$ for $q>0$ are given by
\[
x^{-\nu^i+\sqrt{(\nu^i)^2+\frac{2q}{(\sigma^i)^2}}} \conn x^{-\nu^i-\sqrt{(\nu^i)^2+\frac{2q}{(\sigma^i)^2}}}, \quad x\in\mathcal{I}.\]
The scale function of a diffusion with the generator $\G^i$ is  $-\frac{x^{-2\nu^i}}{2\nu^i}$ for $x\in\mathcal{I}$. By applying \eqref{eq:Green-A-at-alpha}, \eqref{eq:Green-B-at-alpha}, and \eqref{eq:G=GB}, we obtain
\[G_q^A(\alpha,\alpha)=\frac{\alpha^{-2\nu^A}}{\nu^A+\sqrt{(\nu^A)^2+\frac{2q}{(\sigma^A)^2}}}, \quad G_q^B(\alpha,\alpha)=\frac{\alpha^{-2\nu^B}}{-\nu^B+\sqrt{(\nu^B)^2+\frac{2q}{(\sigma^B)^2}}}, \conn G_0^B(\alpha,\alpha)=-\frac{\alpha^{-2\nu^B}}{2\nu^B}.\]
Then, Proposition \ref{prop:laplace-product} produces

\begin{equation*}
	\E^\alpha[e^{-q\lambda_\alpha}]
	=\frac{-2\nu^B\alpha^{-2\nu^A}}
	{\alpha^{-2\nu^A}\cdot \left(-\nu^B+\sqrt{(\nu^B)^2+\frac{2q}{(\sigma^B)^2}}\right) +\alpha^{-2\nu^B}\cdot \left(\nu^A+\sqrt{(\nu^A)^2+\frac{2q}{(\sigma^A)^2}}\right)}.
\end{equation*}
If there is no switch in the parameters, we would have $\mu:=\mu^A=\mu^B$, $\sigma:=\sigma^A=\sigma^B$ with $\nu=\frac{\mu}{\sigma^2}-\frac{1}{2}<0$ and
\[\E^\alpha[e^{-q\lambda_\alpha}]=\frac{-\nu}{\sqrt{\nu^2+\frac{2q}{\sigma^2}}}.\]

Below we illustrate how the switch in parameters affects the distribution of $\lambda_\alpha$ by inverting the Laplace transform to obtain the probability density. Specifically, we focus on the effect arising from the switch in the diffusion parameter. Let us set $\alpha=100$, $\mu^A=\mu^B=-0.1$, $\sigma^A=0.75$, $\sigma^B=1.5$. Thus, the diffusion parameter is higher in the region below $\alpha$, and this setup  is consistent with the leverage effect discussed in the beginning of this subsection. The probability density of the last passage time is given in Figure \ref{fig:inverted_density} where we also present the density for the case of non-switching parameters: $\mu=\mu^A$ and $\sigma=\sigma^A$. The graph in yellow is the last passage time density of the switching parameter diffusion, while the blue one is that of the non-switching diffusion. We drew these graphs by InverseLaplaceTransform function in Wolfram Mathematica 14.0. We see that there is a higher probability of $\lambda_\alpha$ occurring earlier (to the left of the intersection of the two graphs) in the switching parameter case. This is due to the density in the switching parameter case having higher values, compared to the non-switching case, in the lower range of the horizontal axis. This information is useful for risk management.

\begin{figure}[H]
	\caption{\small Probability density of $\lambda_\alpha$ for a geometric Brownian motion with $\alpha=100$. The yellow  density is the  case of switching parameters $\mu^A=\mu^B=-0.1$, $\sigma^A=0.75$, $\sigma^B=1.5$. The blue density is the case of non-switching parameters $\mu=-0.1$, $\sigma=0.75$.}
	\includegraphics[scale=0.9]{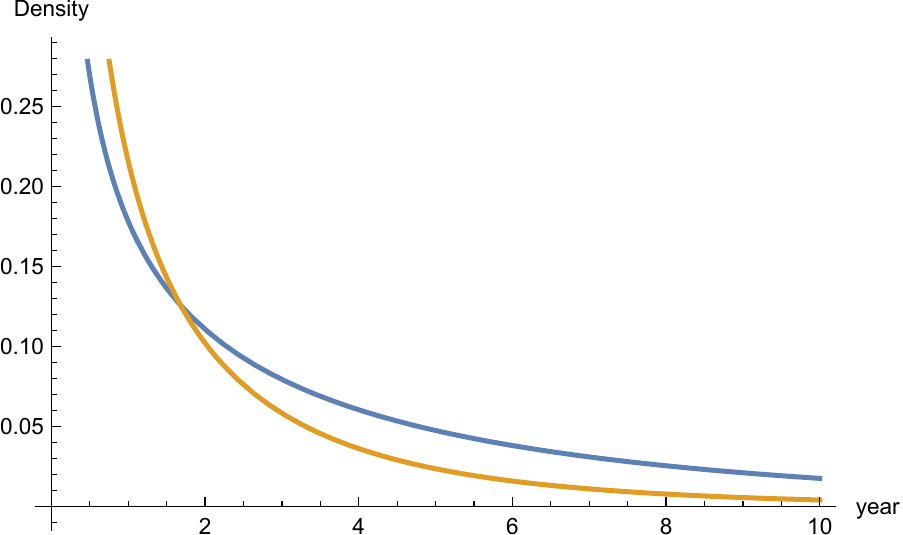}
	\label{fig:inverted_density}
\end{figure}

\setcounter{section}{0}
\appendix
\section{}
\subsection{Elements of a Diffusion}\label{app:elements}
The scale function, speed measure, and killing measure represent basic characteristics of a regular diffusion $X$. The following definitions are based on \citet[Chapter II.1.4]{borodina-salminen}. The scale function $s:\mathcal{I}\to\R$ is an increasing continuous function. If there is no killing inside $\mathcal{I}$, we have the following relation for the first hitting times of $a$ and $b$ (denoted by $H_a$ and $H_b$, respectively):
\begin{equation}\label{eq:scale-def}
\p^x(H_a<H_b)=\frac{s(b)-s(x)}{s(b)-s(a)}, \quad \ell<a\le x\le b<r.
\end{equation}

The speed measure $m(\cdot)$ is a measure on $\mathcal{B}(\mathcal{I})$ such that $m((a,b))\in(0,+\infty)$ for $\ell<a<b<r$. A regular diffusion has a transition density with respect to its speed measure. Specifically, the speed measure satisfies the following relation for every $t>0$ and $x\in \mathcal{I}$:
\begin{align*}
\p^x(X_t\in A)=\int_A P_t(x, \diff y)=\int_Ap(t;x,y)m(\diff y), \quad A\in\B(\mathcal{I})
\end{align*}
where $P_t$ is the transition function and $p(t;x,y)$ denotes the transition density with respect to the speed measure.

The killing measure $k(\cdot)$ is a measure on $\mathcal{B}(\mathcal{I})$ such that $k((a,b))\in[0,+\infty)$ for $\ell<a<b<r$. The killing measure is related to the distribution of the location of the diffusion at its lifetime $	\xi=\inf\{t: X_t\notin\mathcal{I}\}$:
\[\p^x(X_{\xi-}\in A, \xi<t)=\int_0^t\left(\int_A p(s;x,y)k(\diff y)\right)\diff s, \quad A\in\mathcal{B}(\mathcal{I}).\]

We can transform the density $p$ using an excessive function $h$, thus changing the dynamics of $X$. The resulting new diffusion is called \emph{Doob's $h$-transform} of $X$. Refer to \citet[Chapter II.5]{borodina-salminen}. A Borel-measurable function $h: \mathcal{I}\to\R_+$ is called excessive if
\[\E^x[h(X_t)]\le h(x) \quad \textnormal{for all} \; x\in\mathcal{I} \; \textnormal{and} \; t\ge 0,\]
\[\lim_{t\downarrow 0}\E^x[h(X_t)]=h(x) \quad \textnormal{for all} \; x\in\mathcal{I}.\]
We consider a new transition function $P^h_t$ which satisfies
the following equation for $t>0$ and $x\in\mathcal{I}$:
\[P^h_t(x,A)=\int_A\frac{h(y)}{h(x)}p(t;x,y)m(\diff y), \quad A\in\mathcal{B}(\mathcal{I}).\]
This transformation is well-defined because for a regular diffusion $h\equiv 0$ or $h(x)>0$ for all $x\in\mathcal{I}$. The resulting new diffusion $X^h$ with the transition function $P_t^h$ is the $h$-transform of $X$. The scale function and speed measure of $X^h$ can be easily obtained by using the  procedure in \citet[Chapter 15.9]{karlin-book}.

\subsubsection{Analysis of $\psi_q^*$, $\phi_q^*$, $\tilde{\psi}_q$, $\tilde{\phi}_q$}\label{app:case3-analysis} Consider Case 3 and refer to the proof of Proposition \ref{prop:all}. Recall that $\psi_0$, $\psi_q$ are increasing while $\phi_0$, $\phi_q$ are decreasing on $\mathcal{I}$. Then, we see that $\psi_q^*=\frac{\psi_q}{\phi_0}$ is increasing while $\tilde{\phi}_q=\frac{\phi_q}{\psi_0}$ is decreasing on $\mathcal{I}$.

 Next, take any $x\le z$ in $\mathcal{I}$. We have $\p^x(H_z<+\infty)=\p^x(H_z<H_\ell)=\frac{s(x)-s(\ell)}{s(z)-s(\ell)}=\frac{\psi_0(x)}{\psi_0(z)}$ using \eqref{eq:scale-def} and \eqref{eq:case3-phi-psi}. Note that
 $\E^x[e^{-qH_z}]\nearrow \p^x(H_z<+\infty)$ as $q\to 0$. Thus, we obtain with \eqref{eq:hitting-time-laplace}
 \[\frac{\psi_q(x)}{\psi_q(z)}\le \frac{\psi_0(x)}{\psi_0(z)}\]
 from which we see that $\tilde{\psi}_q(x)=\frac{\psi_q(x)}{\psi_0(x)}\le\frac{\psi_q(z)}{\psi_0(z)}=\tilde{\psi}_q(z)$. This proves that $\tilde{\psi}_q$ is increasing on $\mathcal{I}$.

To see that $\phi_q^*$ is a decreasing function, take any $x\ge z$ in $\mathcal{I}$ and observe that $\p^x(H_z<+\infty)=\p^x(H_z<H_r)=\frac{s(r)-s(x)}{s(r)-s(z)}=\frac{\phi_0(x)}{\phi_0(z)}$ using \eqref{eq:scale-def} and \eqref{eq:case3-phi-psi}. Using the same argument as above, by \eqref{eq:hitting-time-laplace} we have
 \[\frac{\phi_q(x)}{\phi_q(z)}\le \frac{\phi_0(x)}{\phi_0(z)}\]
from which we see that $\phi_q^*(x)=\frac{\phi_q(x)}{\phi_0(x)}\le\frac{\phi_q(z)}{\phi_0(z)}=\phi_q^*(z)$. This proves that $\phi_q^*$ is decreasing on $\mathcal{I}$.

\subsection{Proof of Equation \eqref{lemma:psi-phi}}\label{app:proof-phi-psi}
Under Assumption \ref{assump-s}, the functions $\psi_0$ and $\phi_0$ in \eqref{eq:green-0} satisfy the following conditions:
	\begin{equation*}
		\phi_0\equiv1, \quad \psi_0(\ell+)=0, \conn \psi_0(r-)=+\infty.
	\end{equation*}

\begin{proof}
	Let $\ell<x\le y\le z<r$. Then, by \eqref{eq:hitting-time-laplace}
	\begin{align*}\label{eq:boundary-derivations}
		\lim_{q\downarrow 0}\frac{G_q(x,y)}{G_q(y,z)}&=\lim_{q\downarrow 0}\frac{\psi_q(x)\phi_q(y)}{\psi_q(y)\phi_q(z)}=\lim_{q\downarrow 0}\frac{\E^x\left[e^{-qH_y}\right]}{\E^z\left[e^{-qH_y}\right]}
		=\frac{\p^x(H_y<+\infty)}{\p^z(H_y<+\infty)}.
	\end{align*}
	On the other hand, by \eqref{eq:G0} and \eqref{eq:green-0} we obtain
	\begin{equation*}
		\lim_{q\downarrow 0}\frac{G_q(x,y)}{G_q(y,z)}=\frac{G_0(x,y)}{G_0(y,z)}=\frac{\psi_0(x)\phi_0(y)}{\psi_0(y)\phi_0(z)}.
	\end{equation*}
	Hence
	\begin{equation}\label{eq:prob-G0-connection}
		\frac{\p^x(H_y<+\infty)}{\p^z(H_y<+\infty)}=\frac{\psi_0(x)\phi_0(y)}{\psi_0(y)\phi_0(z)}.
	\end{equation}
	For the killing boundary $\ell$, $\lim_{x\downarrow\ell}\p^x(H_y<+\infty)=0$ and we obtain $\psi_0(\ell+)=0$.

As $\p^z(H_y<+\infty)=1$, the right-hand side in \eqref{eq:prob-G0-connection} does not depend on $z$. Thus, the function $\phi_0(z)$ takes the same value for every $z$ and we may set $\phi_0\equiv1$. By substituting $\p^z(H_y<+\infty)=1$ in \eqref{eq:prob-G0-connection}, we also obtain $\psi_0(r-)=+\infty$ due to $\lim_{y\uparrow r}\p^x(H_y<+\infty)=0$.
\end{proof}

\subsection{Analysis of $\psi_q^A$ and $\phi_q^B$}\label{app:increase-decrease}
Let $q>0$. As shown in Section \ref{sec: procedure}, $\psi_q^A=a_1\psi_q+a_2\phi_q$ and $\phi_q^B=b_1\psi_q+b_2\phi_q$ with positive constants $a_1$, $a_2$, $b_1$, $b_2$ given in  \eqref{eq:a1a2} and \eqref{eq:b1b2}. Note that due to the boundary conditions $(\psi_q^A)^+(\alpha)=0$ and $(\phi_q^B)^-(\alpha)=0$, we have  
\[(\psi_q^A)^+(\alpha)=a_1\psi_q'(\alpha)+a_2\phi_q'(\alpha)=0 \conn 
(\phi_q^B)^-(\alpha)=b_1\psi_q'(\alpha)+b_2\phi_q'(\alpha)=0.\]
To prove that $\psi_q^A$ is increasing on $[\alpha,r)$ and $\phi_q^B$ is decreasing on $(\ell,\alpha]$, we consider a function $g(x)=a\psi_q(x)+b\phi_q(x)$ on $\mathcal{I}$ with arbitrary constants $a>0$ and $b>0$. We show that $g$ attains its local minimum at $\alpha$ if $g'(\alpha)=0$ and that there is no local maximum.

Note that $g$ is a positive solution to ODE $\G f=qf$ and
\[\frac{1}{2}\sigma^2(x)g''(x)+\mu(x)g'(x)=qg(x), \quad x\in\mathcal{I}.\]
Thus, $g'(\alpha)=0$ implies $g''(\alpha)>0$ and we have a local minimum at $\alpha$. Since $g''(x)$ should always be positive when $g'(x)=0$, there is no local maximum. By the continuity of $g$, we see that $g$ is decreasing on $(\ell,a]$ and increasing on $[\alpha,r)$.

Finally, we show that $\phi_0^B$ in Section \ref{sec: procedure} is constant on $(\ell,\alpha]$. Using $\phi_0\equiv1$ from \eqref{lemma:psi-phi} and the constants $c_1$, $c_2$ derived in Section \ref{sec: procedure}, we see that $\phi_0^B(x)=c_1\psi_0(x)+c_2\phi_0(x)=0\cdot w_0(s(x)-s(\ell))+1=1$ for $x\in(\ell,\alpha]$.

%\bibliographystyle{elsarticle-harv}
%\bibliography{references}

\end{document}